\documentclass[a4paper]{article}

\usepackage[T1]{fontenc}
\usepackage[utf8]{inputenc}

\usepackage{mathtools,amssymb,amsthm}
\mathtoolsset{mathic}

\usepackage{lmodern}
\DeclareFontFamily{OMS}{mdpgd}{\skewchar\font=48}
\DeclareFontShape{OMS}{mdpgd}{m}{n}{<->s*[1.08]mdpgdr7y}{}
\DeclareMathAlphabet{\mathcal}{OMS}{mdpgd}{m}{n}

\usepackage{tikz-cd}
\tikzcdset{arrow style=Latin Modern, row sep/normal=1.35em, column sep/normal=1.8em, cramped}

\usepackage{enumitem}
\usepackage{etoolbox}
\setlist[itemize,1]{label=\(\vcenter{\hbox{\scriptsize\(\bullet\)}}\)}
\AtBeginEnvironment{conj}{\setlist[enumerate,1]{label={\roman*)}}}
\AtBeginEnvironment{coro}{\setlist[enumerate,1]{label={\roman*)}}}
\AtBeginEnvironment{lemm}{\setlist[enumerate,1]{label={\roman*)}}}
\AtBeginEnvironment{prop}{\setlist[enumerate,1]{label={\roman*)}}}
\AtBeginEnvironment{theo}{\setlist[enumerate,1]{label={\roman*)}}}
\AtBeginEnvironment{conj*}{\setlist[enumerate,1]{label={\roman*)}}}
\AtBeginEnvironment{coro*}{\setlist[enumerate,1]{label={\roman*)}}}
\AtBeginEnvironment{lemm}{\setlist[enumerate,1]{label={\roman*)}}}
\AtBeginEnvironment{prop*}{\setlist[enumerate,1]{label={\roman*)}}}
\AtBeginEnvironment{theo*}{\setlist[enumerate,1]{label={\roman*)}}}
\AtBeginEnvironment{assu}{\setlist[enumerate,1]{label={\roman*)}}}
\AtBeginEnvironment{axio}{\setlist[enumerate,1]{label={\roman*)}}}
\AtBeginEnvironment{cond}{\setlist[enumerate,1]{label={\roman*)}}}
\AtBeginEnvironment{defi}{\setlist[enumerate,1]{label={\roman*)}}}
\AtBeginEnvironment{exem}{\setlist[enumerate,1]{label={\roman*)}}}
\AtBeginEnvironment{exer}{\setlist[enumerate,1]{label={\roman*)}}}
\AtBeginEnvironment{hypo}{\setlist[enumerate,1]{label={\roman*)}}}
\AtBeginEnvironment{prob}{\setlist[enumerate,1]{label={\roman*)}}}
\AtBeginEnvironment{ques}{\setlist[enumerate,1]{label={\roman*)}}}
\AtBeginEnvironment{assu*}{\setlist[enumerate,1]{label={\roman*)}}}
\AtBeginEnvironment{axio*}{\setlist[enumerate,1]{label={\roman*)}}}
\AtBeginEnvironment{cond*}{\setlist[enumerate,1]{label={\roman*)}}}
\AtBeginEnvironment{defi*}{\setlist[enumerate,1]{label={\roman*)}}}
\AtBeginEnvironment{exem}{\setlist[enumerate,1]{label={\roman*)}}}
\AtBeginEnvironment{exer*}{\setlist[enumerate,1]{label={\roman*)}}}
\AtBeginEnvironment{hypo*}{\setlist[enumerate,1]{label={\roman*)}}}
\AtBeginEnvironment{prob*}{\setlist[enumerate,1]{label={\roman*)}}}
\AtBeginEnvironment{ques*}{\setlist[enumerate,1]{label={\roman*)}}}
\AtBeginEnvironment{case}{\setlist[enumerate,1]{label={\roman*)}}}
\AtBeginEnvironment{clai}{\setlist[enumerate,1]{label={\roman*)}}}
\AtBeginEnvironment{conc}{\setlist[enumerate,1]{label={\roman*)}}}
\AtBeginEnvironment{rema}{\setlist[enumerate,1]{label={\roman*)}}}
\AtBeginEnvironment{plan}{\setlist[enumerate,1]{label={\roman*)}}}
\AtBeginEnvironment{merci}{\setlist[enumerate,1]{label={\roman*)}}}
\AtBeginEnvironment{nota}{\setlist[enumerate,1]{label={\roman*)}}}
\AtBeginEnvironment{conv}{\setlist[enumerate,1]{label={\roman*)}}}
\AtBeginEnvironment{convnota}{\setlist[enumerate,1]{label={\roman*)}}}
\AtBeginEnvironment{case*}{\setlist[enumerate,1]{label={\roman*)}}}
\AtBeginEnvironment{clai*}{\setlist[enumerate,1]{label={\roman*)}}}
\AtBeginEnvironment{conc*}{\setlist[enumerate,1]{label={\roman*)}}}
\AtBeginEnvironment{rema*}{\setlist[enumerate,1]{label={\roman*)}}}
\AtBeginEnvironment{plan*}{\setlist[enumerate,1]{label={\roman*)}}}
\AtBeginEnvironment{merci*}{\setlist[enumerate,1]{label={\roman*)}}}
\AtBeginEnvironment{nota*}{\setlist[enumerate,1]{label={\roman*)}}}
\AtBeginEnvironment{conv*}{\setlist[enumerate,1]{label={\roman*)}}}
\AtBeginEnvironment{convnota*}{\setlist[enumerate,1]{label={\roman*)}}}
\AtBeginEnvironment{conj}{\setlist[enumerate,2]{label={\alph*)}}}
\AtBeginEnvironment{coro}{\setlist[enumerate,2]{label={\alph*)}}}
\AtBeginEnvironment{lemm}{\setlist[enumerate,2]{label={\alph*)}}}
\AtBeginEnvironment{prop}{\setlist[enumerate,2]{label={\alph*)}}}
\AtBeginEnvironment{theo}{\setlist[enumerate,2]{label={\alph*)}}}
\AtBeginEnvironment{conj*}{\setlist[enumerate,2]{label={\alph*)}}}
\AtBeginEnvironment{coro*}{\setlist[enumerate,2]{label={\alph*)}}}
\AtBeginEnvironment{lemm}{\setlist[enumerate,2]{label={\alph*)}}}
\AtBeginEnvironment{prop*}{\setlist[enumerate,2]{label={\alph*)}}}
\AtBeginEnvironment{theo*}{\setlist[enumerate,2]{label={\alph*)}}}
\AtBeginEnvironment{assu}{\setlist[enumerate,2]{label={\alph*)}}}
\AtBeginEnvironment{axio}{\setlist[enumerate,2]{label={\alph*)}}}
\AtBeginEnvironment{cond}{\setlist[enumerate,2]{label={\alph*)}}}
\AtBeginEnvironment{defi}{\setlist[enumerate,2]{label={\alph*)}}}
\AtBeginEnvironment{exem}{\setlist[enumerate,2]{label={\alph*)}}}
\AtBeginEnvironment{exer}{\setlist[enumerate,2]{label={\alph*)}}}
\AtBeginEnvironment{hypo}{\setlist[enumerate,2]{label={\alph*)}}}
\AtBeginEnvironment{prob}{\setlist[enumerate,2]{label={\alph*)}}}
\AtBeginEnvironment{ques}{\setlist[enumerate,2]{label={\alph*)}}}
\AtBeginEnvironment{assu*}{\setlist[enumerate,2]{label={\alph*)}}}
\AtBeginEnvironment{axio*}{\setlist[enumerate,2]{label={\alph*)}}}
\AtBeginEnvironment{cond*}{\setlist[enumerate,2]{label={\alph*)}}}
\AtBeginEnvironment{defi*}{\setlist[enumerate,2]{label={\alph*)}}}
\AtBeginEnvironment{exem}{\setlist[enumerate,2]{label={\alph*)}}}
\AtBeginEnvironment{exer*}{\setlist[enumerate,2]{label={\alph*)}}}
\AtBeginEnvironment{hypo*}{\setlist[enumerate,2]{label={\alph*)}}}
\AtBeginEnvironment{prob*}{\setlist[enumerate,2]{label={\alph*)}}}
\AtBeginEnvironment{ques*}{\setlist[enumerate,2]{label={\alph*)}}}
\AtBeginEnvironment{case}{\setlist[enumerate,2]{label={\alph*)}}}
\AtBeginEnvironment{clai}{\setlist[enumerate,2]{label={\alph*)}}}
\AtBeginEnvironment{conc}{\setlist[enumerate,2]{label={\alph*)}}}
\AtBeginEnvironment{rema}{\setlist[enumerate,2]{label={\alph*)}}}
\AtBeginEnvironment{plan}{\setlist[enumerate,2]{label={\alph*)}}}
\AtBeginEnvironment{merci}{\setlist[enumerate,2]{label={\alph*)}}}
\AtBeginEnvironment{nota}{\setlist[enumerate,2]{label={\alph*)}}}
\AtBeginEnvironment{conv}{\setlist[enumerate,2]{label={\alph*)}}}
\AtBeginEnvironment{convnota}{\setlist[enumerate,2]{label={\alph*)}}}
\AtBeginEnvironment{case*}{\setlist[enumerate,2]{label={\alph*)}}}
\AtBeginEnvironment{clai*}{\setlist[enumerate,2]{label={\alph*)}}}
\AtBeginEnvironment{conc*}{\setlist[enumerate,2]{label={\alph*)}}}
\AtBeginEnvironment{rema*}{\setlist[enumerate,2]{label={\alph*)}}}
\AtBeginEnvironment{plan*}{\setlist[enumerate,2]{label={\alph*)}}}
\AtBeginEnvironment{merci*}{\setlist[enumerate,2]{label={\alph*)}}}
\AtBeginEnvironment{nota*}{\setlist[enumerate,2]{label={\alph*)}}}
\AtBeginEnvironment{conv*}{\setlist[enumerate,2]{label={\alph*)}}}
\AtBeginEnvironment{convnota*}{\setlist[enumerate,2]{label={\alph*)}}}

\usepackage[UKenglish]{babel}
\usepackage{microtype}
\usepackage[adjust,nocompress,nosort]{cite}
\usepackage[shortcuts]{extdash}
\usepackage[hidelinks]{hyperref}
\providecommand\phantomsection{}
\usepackage[nottoc]{tocbibind}
\usepackage[useregional]{datetime2}
\usepackage[norefs,nocites]{refcheck}

\theoremstyle{plain}
\newtheorem{conj}{Conjecture}[subsection]
\newtheorem{coro}[conj]{Corollary}
\newtheorem{lemm}[conj]{Lemma}
\newtheorem{prop}[conj]{Proposition}
\newtheorem{theo}[conj]{Theorem}

\newtheorem*{conj*}{Conjecture}
\newtheorem*{coro*}{Corollary}
\newtheorem*{lemm*}{Lemma}
\newtheorem*{prop*}{Proposition}
\newtheorem*{theo*}{Theorem}

\theoremstyle{definition}

\newtheorem{defi}[conj]{Definition}
\newtheorem{exem}[conj]{Example}

\newtheorem*{assu*}{Assumption}
\newtheorem*{axio*}{Axiom}
\newtheorem*{cond*}{Condition}
\newtheorem*{defi*}{Definition}
\newtheorem*{exem*}{Example}
\newtheorem*{exer*}{Exercise}
\newtheorem*{hypo*}{Hypothesis}
\newtheorem*{prob*}{Problem}
\newtheorem*{ques*}{Question}

\theoremstyle{remark}

\newtheorem{rema}[conj]{Remark}

\newtheorem*{case*}{Case}
\newtheorem*{clai*}{Claim}
\newtheorem*{conc*}{Conclusion}
\newtheorem*{rema*}{Remark}
\newtheorem*{plan*}{Organisation of the article}
\newtheorem*{merci*}{Acknowledgements}
\newtheorem*{nota*}{Notations}
\newtheorem*{conv*}{Conventions}
\newtheorem*{convnota*}{Conventions \& notations}

\renewcommand{\theconj}{%
	\ifnum\value{section}=0
		\arabic{conj}%
	\else
		\ifnum\value{subsection}=0
			\thesection
		\else
			\thesubsection
		\fi
			.\arabic{conj}%
	\fi
}

\numberwithin{equation}{subsection}
\renewcommand{\theequation}{%
	\ifnum\value{section}=0
		\arabic{equation}%
	\else
		\ifnum\value{subsection}=0
			\thesection
		\else
			\thesubsection
		\fi
			.\arabic{equation}%
	\fi
}

\providecommand{\dedicace}[1]{\begin{flushright}\textit{#1}\end{flushright}}
\providecommand{\msc}[1]{{\noindent\small\textbf{Mathematics Subject Classification (2020)} --- #1.}}
\providecommand{\keywords}[1]{{\noindent\small\textbf{Keywords} --- #1.}}

\title{\textbf{Universal norms and the Fargues-Fontaine curve}}
\author{Gautier Ponsinet}
\date{5th October 2020}

\providecommand{\contact}{{
	\bigskip
	\small

	\noindent
	\textsc{G.~Ponsinet}, Max Planck Institut für Mathematik, Vivatsgasse 7, 53111 Bonn, Germany
	\par\noindent\nopagebreak
	\textit{E-mail address:} \href{mailto:gautier.ponsinet@mpim-bonn.mpg.de}{\url{gautier.ponsinet@mpim-bonn.mpg.de}}
}}

\newcommand{\ie}{\textit{i.e.} }

\let\leq\leqslant
\let\geq\geqslant
\newcommand{\ra}{\rightarrow}

\newcommand{\xra}[1]{\xrightarrow{#1}}
\newcommand{\riso}{\xra{\raisebox{-0.5ex}[0ex][0ex]{\(\sim\)}}}

\newcommand{\ul}[1]{\underline{#1}}
\newcommand{\h}[1]{\hat{#1}}
\newcommand{\dprime}{{\prime\prime}}

\newcommand{\N}{\mathbf{N}}
\newcommand{\Z}{\mathbf{Z}}

\newcommand{\Zp}{\mathbf{Z}_p}
\newcommand{\Qp}{\mathbf{Q}_p}
\newcommand{\Cp}{\mathbf{C}_p}

\newcommand{\Qpbar}{\bar{\mathbf{Q}}_p}

\renewcommand{\AA}{\mathcal{A}}

\newcommand{\CC}{\mathcal{C}}

\newcommand{\EE}{\mathcal{E}}
\newcommand{\FF}{\mathcal{F}}
\newcommand{\GG}{\mathcal{G}}
\newcommand{\HH}{\mathcal{H}}
\newcommand{\II}{\mathcal{I}}

\newcommand{\LL}{\mathcal{L}}
\newcommand{\MM}{\mathcal{M}}

\newcommand{\OO}{\mathcal{O}}

\renewcommand{\SS}{\mathcal{S}}

\newcommand{\mg}{\mathfrak{m}}

\DeclareMathOperator{\Hom}{Hom}
\DeclareMathOperator{\Ker}{Ker}
\DeclareMathOperator{\Img}{Im}
\DeclareMathOperator{\Coker}{Coker}

\DeclareMathOperator{\Sym}{Sym}
\DeclareMathOperator{\Obj}{Ob}

\let\temp\phi
\let\phi\varphi
\let\varphi\temp
\let\temp\epsilon
\let\epsilon\varepsilon
\let\varepsilon\temp
\newcommand{\liminj}{\varinjlim}
\newcommand{\limproj}{\varprojlim}

\DeclareMathOperator{\Gal}{Gal}

\DeclareMathOperator{\Spec}{Spec}

\DeclareMathOperator{\Proj}{Proj}

\newcommand{\et}{\text{ét}}

\newcommand{\BdR}{\mathbf{B}_\mathrm{dR}}

\newcommand{\Bcris}{\mathbf{B}_\mathrm{cris}}
\newcommand{\Be}{\mathbf{B}_\mathrm{e}}
\newcommand{\DdR}{\mathbf{D}_\mathrm{dR}}

\DeclareMathOperator{\Fil}{Fil}

\DeclareFontFamily{U}{wncy}{}
\DeclareFontShape{U}{wncy}{m}{n}{<->wncyr10}{}
\DeclareSymbolFont{mcy}{U}{wncy}{m}{n}
\DeclareMathSymbol{\Sha}{\mathord}{mcy}{"58}
\newcommand{\Crm}{\mathrm{C}}
\newcommand{\Brm}{\mathrm{B}}
\newcommand{\Zrm}{\mathrm{Z}}

\renewcommand{\H}{\mathrm{H}}
\newcommand{\Iw}{\mathrm{Iw}}
\newcommand{\B}{\mathbf{B}}
\newcommand{\Mod}{\mathrm{Mod}}
\newcommand{\Rep}{\mathrm{Rep}}
\newcommand{\XFF}{X^\mathrm{FF}}
\newcommand{\e}{\mathrm{e}}
\newcommand{\dR}{\mathrm{dR}}
\newcommand{\Coh}{\mathrm{Coh}_\mathrm{FF}}
\newcommand{\Bun}{\mathrm{Bun}_\mathrm{FF}}
\newcommand{\fr}{\mathrm{free}}
\DeclareMathOperator{\rk}{rk}
\DeclareMathOperator{\length}{length}
\newcommand{\RGamma}{\mathrm{R}\Gamma}
\renewcommand{\Vec}{\mathrm{Vec}}
\newcommand{\fl}{\mathrm{flat}}
\newcommand{\trunc}{\tau^{\leq 0}}
\newcommand{\id}{\mathrm{id}}
\newcommand{\HT}{\mathrm{HT}}
\newcommand{\disc}{{\mathrm{disc}}}
\newcommand{\discu}{{\ul{\mathrm{disc}}}}
\newcommand{\res}{\mathrm{res}}
\newcommand{\cor}{\mathrm{cor}}
\newcommand{\norm}{\mathrm{norm}}
\newcommand{\D}{\mathbf{D}}
\newcommand{\cris}{\mathrm{cris}}
\DeclareMathOperator{\Gr}{Gr}
\newcommand{\Top}{\mathrm{Top}}
\newcommand{\op}{\mathrm{op}}

\begin{document}
\maketitle
\dedicace{À Alma}
\smallskip
\begin{abstract}
	We study the module of universal norms associated with a de Rham \(p\)\=/adic Galois representation in a perfectoid field extension.
	In particular, we compute precisely this module when the Hodge-Tate weights of a representation are greater than or equal to \(0\).
	This generalises a result by Coates and Greenberg for Abelian varieties, and partially answers a question of theirs.
	Our method relies on the classification of vector bundles over the Fargues-Fontaine curve.
\end{abstract}
\bigskip
\newline
\msc{Primary: 11R23; Secondary: 11F80, 11F85, 11S25, 14F30, 14G45, 14H60}
\tableofcontents

\section*{Introduction} \label{sec:intro}
\phantomsection
\addcontentsline{toc}{section}{\nameref{sec:intro}}
Let \(p\) be a prime number.
In the present article, we are interested in local Iwasawa theory of \(p\)\=/adic Galois representations in perfectoid field extensions.

Let \(K\) be a finite extension of the field \(\Qp\) of \(p\)\=/adic numbers, contained in some algebraic closure \(\Qpbar\) of \(\Qp\), and let \(\GG_K = \Gal(\Qpbar/K)\) be the absolute Galois group of \(K\).
Let \(V\) be a \(p\)\=/adic representation of \(\GG_K\), \ie \(V\) is a finite dimensional \(\Qp\)\=/vector space endowed with a continuous and linear action of \(\GG_K\), and let \(T\) be a \(\Zp\)\=/lattice in \(V\) stable under the action of \(\GG_K\).
Let \(L\) be an algebraic extension of \(K\) contained in \(\Qpbar\).
We consider  the first Galois cohomology group
\[
	\H^1(L,V/T) = \H^1(\Gal(\Qpbar/L),V/T).
\]
Bloch and Kato~\cite{BlochKato90} have defined arithmetically significant subgroups
\[
	\H^1_e(L,V/T) \subset \H^1_f(L,V/T) \subset \H^1_g(L,V/T) \subset \H^1(L,V/T),
\]
which are involved in their conjecture on special values of \(L\)\=/functions of motifs.
For instance, if \(T=T_p(A)\) is the \(p\)\=/adic Tate module of an Abelian variety \(A\) defined over \(K\), so that \(V/T\) is the subgroup \(A[p^\infty]\) of \(p\)\=/power torsion points of \(A(\Qpbar)\), then the Bloch-Kato subgroups of \(\H^1(L,A[p^\infty])\) are all equal and coincide with the image of the Kummer map
\[
	\kappa_L : A(L) \otimes_{\Z} \Qp/\Zp \ra \H^1(L,A[p^\infty]).
\]

A precise description of the Bloch-Kato subgroups when \(L\) is an infinite extension of \(K\) is essential to Iwasawa theory.
Indeed, such a description makes for example possible the study of Selmer groups of motifs over infinite extensions~\cite{Mazur72,Greenberg03,CoatesHowson01} and the construction of \(p\)\=/adic height pairings~\cite{Schneider82,PerrinRiou92,Nekovar93}.

Perfectoid fields, a class of complete non-Archimedean fields introduced by Scholze~\cite{Scholze12}, are of particular interest for Iwasawa theory and its non-commutative generalisation~\cite{CFKSV05}.
Let \(\h{L}\) be the completion of \(L\) for the \(p\)\=/adic valuation topology.
First, we recall that if \(\h{L}\) is perfectoid then the extension \(L/K\) is infinite and infinitely wildly ramified.
Furthermore, by a theorem of Sen~\cite{Sen72}, if \(L\) is an infinite Galois extension of \(K\) with finite residue field and whose Galois group is a \(p\)\=/adic Lie group, then \(\h{L}\) is perfectoid.
Such extensions are ubiquitous in Iwasawa theory.
In particular, the completion of the cyclotomic extension of \(K\) generated by all the \(p\)\=/power roots of unity is a perfectoid field.

Coates and Greenberg~\cite{CoatesGreenberg96} gave the following simple cohomological description of the image of the Kummer map \(\kappa_L\) associated with an Abelian variety when \(\h{L}\) is a perfectoid field.
If the representation \(V\) is de Rham, a class of \(p\)\=/adic representations defined by Fontaine~\cite{Fontaine94II}, we set \(V_0\) the minimal sub\=/\(\GG_K\)\=/representation of \(V\) such that the Hodge-Tate weights of \(V/V_0\) are all less than or equal to \(0\), and \(T_0 = T \cap V_0\).
The inclusion \(V_0/T_0 \subset V/T\) induces a map
\[
	\lambda_L : \H^1(L,V_0/T_0) \ra \H^1(L,V/T).
\]
If \(A\) is an Abelian variety defined over \(K\) and \(T=T_p(A)\), then the representation \(V=\Qp \otimes_{\Zp} T_p(A)\) is de Rham with Hodge-Tate weights \(0\) and \(1\), and Coates and Greenberg proved\footnote{We have slightly rephrased their result here, see Remark~\ref{rema:CGcomparaison}.} that if \(\h{L}\) is a perfectoid field, then
\begin{equation} \label{eq:CGtheo_intro}
	\Img(\kappa_L) = \Img(\lambda_L).
\end{equation}

The purpose of the present article is to study the following question asked by Coates and Greenberg~\cite[p.~131]{CoatesGreenberg96}.
\begin{ques*}[Coates \& Greenberg]
	Does a description of the Bloch-Kato subgroups analogous to~\eqref{eq:CGtheo_intro} exist when \(V\) is a general \(p\)\=/adic representation?
\end{ques*}

Berger~\cite{Berger05}, generalising works of Perrin-Riou~\cite{PerrinRiou92,PerrinRiou00,PerrinRiou01}, answered this question positively when \(L\) is the cyclotomic extension of \(K\) and \(V\) is de Rham.
However, both Berger and Perrin-Riou's approaches use tools specific to the cyclotomic extension, and therefore seem difficult to extend to other extensions.

The first and most precise result of the present article regarding Coates and Greenberg's question is the following generalisation of~\eqref{eq:CGtheo_intro}.
\begin{theo} \label{theo:intro_precis}
	Assume \(V\) is de Rham with Hodge-Tate weights less than or equal to \(1\).
	If \(\h{L}\) is a perfectoid field, then
	\[
		\H^1_e(L,V/T) = \Img(\lambda_L).
	\]
\end{theo}

The study of the Bloch-Kato subgroups over infinite extensions, motivated by Iwasawa theory, has a long history.
For Abelian varieties and \(p\)\=/divisible groups when \(L/K\) is a ramified \(\Zp\)\=/extension, it has been initiated by Mazur~\cite{Mazur72} and further developed in an important literature~\cite{Manin71,Vvedenskii73,Konovalov76,Hazewinkel74,Hazewinkel74II,Hazewinkel77,Hazewinkel78,LubinRosen78,Rubin85,Schneider87}.
Coates and Greenberg's theorem~\eqref{eq:CGtheo_intro} (and therefore, Theorem~\ref{theo:intro_precis}) generalises all these results.

We recall that, for an Abelian variety \(A\) defined over \(K\), to compute the image of the Kummer map \(\kappa_L\) is equivalent, by Tate local duality, to compute the module of universal norms associated with the dual Abelian variety
\begin{equation} \label{eq:norms_intro}
	\limproj_{\mathrm{norm}} A^\vee(K^\prime)
\end{equation}
where \(K^\prime\) runs through the finite extensions of \(K\) contained in \(L\).
Prior to the work of Coates and Greenberg (whose approach we shall recall briefly below), this is this last module~\eqref{eq:norms_intro} of universal norms that was computed directly.
We shall deduce from Theorem~\ref{theo:intro_precis} a description of the module of universal norms of Bloch-Kato subgroups (see Theorems~\ref{theo:Iwasawa}, \ref{theo:Iwasawafini} and~\ref{theo:Iwasawa_autresBK}).

We now explain our proof of Theorem~\ref{theo:intro_precis}.
Fontaine~\cite{Fontaine03} has defined a discrete \(\GG_K\)\=/module \(E_\discu(V/T)\) containing \(V/T\) as torsion subgroup, and whose first Galois cohomology group fits into a short exact sequence
\begin{equation} \label{eq:Edisc_intro}
		0 \ra \H^1_e(L,V/T) \ra \H^1(L,V/T) \ra \H^1(L,E_\discu(V/T)) \ra 0.
\end{equation}
Precisely, \(E_\discu(V/T)\) is the maximal subgroup stable under the action of \(\GG_K\) of \((\Be \otimes_{\Qp} V)/T\) on which the action of \(\GG_K\) is discrete, endowed with the discrete topology, where \(\Be = \Bcris^{\phi = 1}\) and \(\Bcris\) is Fontaine's ring of crystalline \(p\)\=/adic periods.

If \(A\) is an Abelian variety defined over \(K\), then \(E_\discu(A[p^\infty])\) is canonically isomorphic to \(A(\Qpbar)\) up to prime-to\=/\(p\) torsion, and the short exact sequence~\eqref{eq:Edisc_intro} is isomorphic to the Kummer short exact sequence
\[
	0 \ra A(L) \otimes_{\Z} \Qp/\Zp \xra{\kappa_L} \H^1(L,A[p^\infty]) \ra \H^1(L,A)[p^\infty] \ra 0.
\]
Coates and Greenberg use explicit methods to compute \(\H^1(L,A)[p^\infty]\) when \(\h{L}\) is perfectoid, and deduce their result~\eqref{eq:CGtheo_intro}.
However, this explicit approach seems difficult to generalise to \(\H^1(L,E_\discu(V/T))\).

Instead, the subspace topology from \((\Be \otimes_{\Qp} V)/T\) endows the module \(E_\discu(V/T)\) with a linear and separated topology, and we consider its Hausdorff completion \(E_+(V/T)\) also introduced by Fontaine~\cite{Fontaine03}.
For instance, if \(A\) is an Abelian variety defined over \(K\), this topology on \(E_\discu(A[p^\infty])\) coincides with the natural topology of \(A(\Qpbar)\) via the aforementioned canonical isomorphism, and the completion \(E_+(A[p^\infty])\) is isomorphic to \(A(\Cp)\) up to prime-to\=/\(p\) torsion, where \(\Cp\) is the completion of \(\Qpbar\) for the \(p\)\=/adic valuation topology.

The completion map \(E_\discu(V/T) \ra E_+(V/T)\) induces an application
\begin{equation} \label{eq:completion}
	\xi_L : \H^1(L,E_\discu(V/T)) \ra \H^1(L,E_+(V/T)).
\end{equation}
The most general result of the present article regarding Coates and Greenberg's question is then the following.
\begin{theo} \label{theo:intro_principal}
	Assume \(V\) is de Rham.
	If \(\h{L}\) is a perfectoid field, then there is an isomorphism
	\[
		\frac{\Img(\lambda_L)}{\H^1_e(L,V/T)} \simeq \Ker\left(\H^1(L,E_\discu(V/T)) \xra{\xi_L} \H^1(L,E_+(V/T))\right).
	\]
\end{theo}

Under the additional assumption of Theorem~\ref{theo:intro_precis}, we shall prove that the map~\eqref{eq:completion} is an isomorphism, and then Theorem~\ref{theo:intro_principal} implies Theorem~\ref{theo:intro_precis}.

To prove Theorem~\ref{theo:intro_principal}, we shall demonstrate that if \(\h{L}\) is perfectoid, then there is a natural isomorphism
\[
	\H^1(L,E_+(V/T)) \simeq \H^1(L,(V/V_0)/(T/T_0)),
\]
which generalise results of Coates and Greenberg for Abelian varieties and \(p\)\=/divisible groups (see Corollary~\ref{coro:perfectoidplus} and Remark~\ref{rema:byproductEdisc}).
To do so, we use the classification of vector bundles over the Fargues-Fontaine curve~\cite{FarguesFontaine18}.
We recall that the Fargues-Fontaine curve, which we denote by \(\XFF\), is a projective scheme build out of Fontaine's rings of \(p\)\=/adic periods, and whose geometry is closely related to \(p\)\=/adic Hodge theory.
Let \(E_+(V)= \limproj_{p\times} E_+(V/T)\).
Then \(E_+(V)\) is an almost \(\Cp\)\=/representation of \(\GG_K\), a category introduced by Fontaine~\cite{Fontaine03}, and there exists a \(\GG_K\)\=/equivariant vector bundle \(\EE_+(V)\) over \(\XFF\) such that
\[
	E_+(V) \simeq \Gamma(\XFF,\EE_+(V)).
\]
Fargues and Fontaine's classification of vector bundles over \(\XFF\) combined with a study of the vector bundle \(\EE_+(V)\) enable the computation of the group \(\H^1(L,E_+(V))\) when \(\h{L}\) is a perfectoid field, and consequently, the computation of \(\H^1(L,E_+(V/T))\).

\begin{plan*}
	In the first section, we review the necessary material on coherent sheaves over the Fargues-Fontaine curve from~\cite{FarguesFontaine18}.
	We also include a reminder on Fontaine's rings of \(p\)\=/adic periods.
	This first section, which contains no new results, is also used to fix notations for the rest of the article.
	In the second section, we study the vector bundle \(\EE_+(V)\).
	Instead of a direct examination of \(\EE_+(V)\), we define and study a modification of de Rham vector bundles over the curve \(\XFF\), of which the vector bundle \(\EE_+(V)\) is a particular case.
	In the third and final section, we first recall the definition of the groups \(E_\discu(V/T)\) and \(E_+(V/T)\). Then we relate these groups to the Bloch-Kato subgroups and the vector bundle \(\EE_+(V)\) respectively. Finally, we deduce from these relations the main results of the article: Theorems~\ref{theo:intro_precis} and~\ref{theo:intro_principal}, as well as various corollaries.
	To prove that the map~\eqref{eq:completion} is an isomorphism under the assumption of Theorem~\ref{theo:intro_precis}, we shall use the fact that Galois cohomology groups can be equipped with a well-behaved structure of topological group, this is reviewed in the appendix.
\end{plan*}

\begin{merci*}
	I thank Kâzım Büyükboduk and Arthur-César Le Bras for helpful discussions.
	I also thank Antonio Lei for his comments on a preliminary version of this text.
	Part of this work was done while visiting the Korean Institute for Advanced Study.
	I thank the KIAS and Chan-Ho Kim for this opportunity and their hospitality.
	I am grateful to the Max Planck Institute for Mathematics, for hospitality, support, and excellent working conditions.
	As it will be clear to the reader, the article of Coates and Greenberg~\cite{CoatesGreenberg96} as well as Fontaine's works~\cite{Fontaine03,Fontaine20} and~\cite{FarguesFontaine18} with Fargues have been major influences to this work and constant sources of inspiration.
\end{merci*}

\bigskip
\begin{convnota*}
	We adopt the convention that the set of natural numbers \(\N\) contains \(0\).

	We fix once and for all a prime number \(p\), an algebraic closure \(\Qpbar\) of the field \(\Qp\) of \(p\)\=/adic numbers, and a finite extension \(K\) of \(\Qp\) contained in \(\Qpbar\).
	We denote by \(\GG_K=\Gal(\Qpbar/K)\) the absolute Galois group of \(K\), and by \(K_0\) the maximal unramified extension of \(\Qp\) contained in \(K\).
	The completion of \(\Qpbar\) for the \(p\)\=/adic valuation topology is denoted by \(\Cp\).
	The action of \(\GG_K\) on \(\Qpbar\) extends by continuity to \(\Cp\).

	Every algebraic extension of \(K\) considered throughout the article will be contained in \(\Qpbar\).
	If \(L\) is such an extension, then we denote by \(\GG_L=\Gal(\Qpbar/L)\) the absolute Galois group of \(L\), by \(L_0\) the maximal unramified extension of \(\Qp\) contained in \(L\), and by \(\h{L}\) the completion of \(L\) for the \(p\)\=/adic valuation topology.

	If \(F\) is a valued field, we denote by \(\OO_F\) its valuation ring, by \(\mg_F\) the maximal ideal of \(\OO_F\), and by \(k_F = \OO_F/\mg_F\) the residue field of \(F\).

	If \(G\) is a topological group and \(M\) is a topological \(G\)\=/module, then, for all \(n \in \N\), we denote by \(\H^n(G,M)\) the \(n\)\=/th group of continuous group cohomology of \(G\) with coefficients in \(M\) (see Appendix~\ref{apdx:contcohom}).
	If \(G=\GG_k\) is the absolute Galois group of some field \(k\),
	then we shall write \(\H^n(k,M)\) instead of \(\H^n(\GG_k,M)\), and \(\H^n(k,M)\) is the \(n\)\=/th group of Galois cohomology of \(\GG_k\) with coefficients in \(M\).

	We denote by \(\Zp(1)\) the free \(\Zp\)\=/module of rank \(1\) whose elements are sequences \((\zeta_{p^n})_{n \in \N}\) of \(p\)\=/power roots of unity in \(\Qpbar\) such that \(\zeta_1 = 1\) and \(\zeta_{p^{n+1}}^p = \zeta_{p^n}\) for each \(n \in \N\), endowed with the natural action of \(\GG_K\).
	We fix a generator \(t\) of \(\Zp(1)\) with group law written additively.
	For all \(n \in \N\), we set
	\[
		\begin{split}
			\Zp(n)  & = \Sym^n_{\Zp}(\Zp(1)),\\
			\Zp(-n) & = \Hom_{\Zp}(\Zp(n),\Zp).
		\end{split}
	\]
	If \(N\) is a \(\Zp\)\=/module equipped with a linear action of \(\GG_K\), then, for all \(n \in \Z\), we set
	\[
		N(n) = N \otimes_{\Zp} \Zp(n).
	\]
\end{convnota*}

\section{Coherent sheaves over the Fargues-Fontaine curve} \label{sec:FF}
In this section, we recall the results on coherent sheaves over the Fargues-Fontaine curve~\cite{FarguesFontaine18} which shall be needed.
We first give a brief overview of the \(p\)\=/adic periods rings introduced by Fontaine~\cite{Fontaine94}.

\subsection{\texorpdfstring{\(p\)\=/}{p-}adic periods rings} \label{subsec:periods}
The ring \(\BdR^+\) is a complete discrete valuation ring endowed with an action of \(\GG_K\), containing \(\Zp(1)\) as sub\=/\(\GG_K\)\=/modules, whose residue field is \(\Cp\), and of which \(t\) is a uniformiser.
Moreover, \(\BdR^+\) has a structure of \(K\)\=/algebra, and there exists an isomorphism between the separable closure of \(K\) in \(\BdR^+\) and \(\Qpbar\) compatible with the action \(\GG_K\) which we use to identify these two fields.
The \emph{field of \(p\)\=/adic periods} \(\BdR\) is the field of fractions of \(\BdR^+\).
There is a natural filtration on \(\BdR\) by the fractional ideals
\[
	\Fil^n \BdR = \BdR^+ \cdot t^n, n \in \Z,
\]
which is stable under the action of \(\GG_K\).
For each \(n \in \N\), we set \(\B_n = \BdR^+ / \Fil^n \BdR\).
In particular, we have \(\B_0 = 0\) and \(\B_1 = \Cp\).

The field \(\BdR\) is equipped with a topology, the \emph{canonical topology}, which is coarser than the valuation topology from \(\BdR^+\).
The action of \(\GG_K\) on \(\BdR\) endowed with the canonical topology is continuous, and we have
\begin{equation} \label{eq:FilGalois}
	\begin{aligned}
		\BdR^{\GG_K} = (\Fil^n \BdR)^{\GG_K} & = K & \text{if } n \leq 0,\\
		(\Fil^n \BdR)^{\GG_K} & = 0 & \text{if } n > 0.
	\end{aligned}
\end{equation}
Moreover, the ring \(\BdR^+\) is the topological closure of \(\Qpbar\) in \(\BdR\) relatively to the canonical topology~\cite{Colmez12}.
Unless otherwise stated, we will consider \(\BdR\) and its subrings and subquotient rings endowed with the canonical topology.
In particular, the topology on \(\Cp=\B_1\) coincides with the usual \(p\)\=/adic valuation topology.

The ring \(\BdR^+\) contains a sub\=/\(K_0\)\=/algebra \(\Bcris^+\), stable under the action of \(\GG_K\), containing \(\Zp(1)\), and equipped with a continuous and semi-linear (with respect to the absolute Frobenius acting on \(K_0\)) endomorphism \(\phi\) commuting with the action of \(\GG_K\).
The \emph{ring of crystalline periods} is \(\Bcris=\Bcris^+[1/t]\).
We have \(\Bcris^{\GG_K} = K_0\).
Moreover, we have \(\phi(t) = pt\), and the endomorphism \(\phi\) extends uniquely to \(\Bcris\).

Let
\[
	\Be = \Bcris^{\phi = 1} = \{ b \in \Bcris, \phi(b)=b \},
\]
and, for each \(n \in \N\), let
\[
	(\Bcris^+)^{\phi = p^n} = \{ b \in \Bcris^+, \phi(b)=p^n b \}.
\]

The ring \(\Be\) is a principal ideal domain~\cite[Théorème~6.5.2]{FarguesFontaine18}.
It inherits a filtration from \(\BdR\), and we have
\[
	\Fil^n \Be = \Fil^n \BdR \cap \Be = \left\{ \begin{array}{ll}
		(\Bcris^+)^{\phi = p^n} \cdot t^{-n} & \text{if } n \leq 0,\\
		0 & \text{if } n > 0.
	\end{array}\right.
\]
Furthermore, we have \(\Fil^0 \Be = (\Bcris^+)^{\phi = 1} = \Qp\), and there exists a \(\GG_K\)\=/equivariant short exact sequence of topological \(\Qp\)\=/algebras, the so-called \emph{fundamental exact sequence},
\begin{equation} \label{eq:funda}
	0 \ra \Qp \ra \Be \ra \BdR/\BdR^+ \ra 0,
\end{equation}
where the map \(\Be \ra \BdR/\BdR^+\) is the composition of the projection of \(\BdR\) on \(\BdR/\BdR^+\) with the inclusion of \(\Be\) in \(\BdR\).

For \(Z \in \{\Qp, \Be,\BdR^+,\BdR\}\), we denote by \(\Mod_Z\) the category of finitely generated \(Z\)\=/modules.
The topology of \(Z\) induces a natural topology on each finitely generated \(Z\)\=/module.
If \(L\) is an algebraic extension of \(K\) (possibly \(L=K\)), then we denote by \(\Rep_Z(\GG_L)\) the subcategory of \(\Mod_Z\) of finitely generated \(Z\)\=/modules equipped with a continuous and semi-linear action of \(\GG_L\), and whose morphisms are \(\GG_L\)\=/equivariant \(Z\)\=/linear applications.
An object of \(\Rep_Z(\GG_L)\) is a \emph{\(Z\)\=/representation of \(\GG_L\)}, and a \(\Qp\)\=/representation of \(\GG_L\) is simply called a \emph{\(p\)\=/adic representation of \(\GG_L\)}.

\subsection{The Fargues-Fontaine curve} \label{subsec:FFcurve}
The \emph{Fargues-Fontaine curve} is the scheme
\[
	\XFF = \Proj\left( \bigoplus_{n \in \N} (\Bcris^+)^{\phi = p^n} \right).
\]
It is a regular, Noetherian, separated and connected \(1\)\=/dimensional scheme defined over \(\Qp\) with global sections \(\Gamma(\XFF,\OO_{\XFF}) = \Qp\).
Moreover, the curve \(\XFF\) is complete, \ie the divisor of a non-zero rational function on \(\XFF\) has degree zero.

The Galois group \(\GG_K\) naturally acts continuously on \(\XFF\) and there exists a unique closed point of \(\XFF\), denoted by \(\infty\), fixed under this action.
The completion of the fibre \(\OO_{\XFF,\infty}\) is canonically isomorphic to \(\BdR^+\), and we have \(\XFF \setminus \{\infty\} = \Spec(\Be)\).
The glueing of \(\XFF \setminus \{\infty\}\) and \(\{\infty\}\) defining \(\XFF\) is done via the natural inclusion \(\Be \subset \BdR\).

\subsection{Coherent sheaves} \label{subsec:Coh}
We recall the description à la Beauville-Laszlo of coherent sheaves over \(\XFF\).
Let \(\MM\) be the fibre product of the categories
\[
	\begin{tikzcd}
							& \Mod_{\BdR^+} \ar{d} \\
		\Mod_{\Be} \ar{r}	& \Mod_{\BdR},
	\end{tikzcd}
\]
where the arrows are the functors of extension of scalars.
An object of \(\MM\) is a triple \((\FF_\e,\FF_\dR^+,\iota_{\FF})\) with
\begin{itemize}
	\item \(\FF_\e\) a finitely generated \(\Be\)\=/module,
	\item \(\FF_\dR^+\) a finitely generated \(\BdR^+\)\=/module,
	\item and
		\[
			\iota_{\FF} : \BdR \otimes_{\Be} \FF_\e \ra \BdR \otimes_{\BdR^+} \FF_\dR^+
		\]
		an isomorphism of \(\BdR\)\=/vector spaces.
\end{itemize}
A morphism between objects of \(\MM\)
\[
	(\FF_\e, \FF_\dR^+, \iota_{\FF}) \ra (\HH_\e, \HH_\dR^+, \iota_{\HH})
\]
is a pair \((f_\e, f_\dR^+)\) with
\begin{itemize}
	\item \(f_\e : \FF_\e \ra \HH_\e \) a morphism of \(\Be\)\=/modules,
	\item \(f_\dR^+ : \FF_\dR^+ \ra \HH_\dR^+ \) a morphism of \(\BdR^+\)\=/modules,
\end{itemize}
such that the diagram
\[
	\begin{tikzcd}
		\BdR \otimes_{\Be} \FF_\e \ar{r}{\iota_{\FF}} \ar{d}{1 \otimes f_\e} & \BdR \otimes_{\BdR^+} \FF_\dR^+ \ar{d}{1 \otimes f_\dR^+}\\
		\BdR \otimes_{\Be} \HH_\e \ar{r}{\iota_{\HH}}                       & \BdR \otimes_{\BdR^+} \HH_\dR^+
	\end{tikzcd}
\]
commutes.

Let \(\Coh\) be the category of coherent sheaves over \(\XFF\).
There is a functor from \(\Coh\) to \(\MM\) which associates with a coherent sheaf \(\FF\) the object of \(\MM\) defined by
\begin{itemize}
	\item \(\FF_\e = \Gamma(\XFF\setminus \{\infty\}, \FF)\),
	\item \(\FF_\dR^+ = \BdR^+ \otimes_{\OO_{\XFF,\infty}} \FF_\infty\),
	\item and \(\iota_{\FF}\) the natural glueing data map.
\end{itemize}

\begin{exem}
	With the structure sheaf \(\OO_{\XFF}\) is associated the triple
	\[
		(\Be, \BdR^+, \iota_{\OO_{\XFF}})
	\]
	where the map \(\iota_{\OO_{\XFF}}\) is induced by the natural inclusion \(\Be \subset \BdR\).
\end{exem}

\begin{theo} \label{theo:BeauvilleLaszlo}
	The functor
	\[
		\begin{split}
			\Coh & \ra \MM \\
			\FF  & \mapsto (\FF_\e, \FF_\dR^+, \iota_{\FF})
		\end{split}
	\]
	is an equivalence of categories.
\end{theo}

Under the equivalence of Theorem~\ref{theo:BeauvilleLaszlo}, the full subcategory \(\Bun\) of \(\Coh\) of \emph{vector bundle} over \(\XFF\), \ie of torsion free coherent \(\OO_{\XFF}\)\=/modules, is equivalent to the full subcategory \(\MM^\fr\) of \(\MM\) whose objects are triple \((\FF_\e, \FF_\dR^+, \iota_{\FF})\) such that both the \(\Be\)\=/module \(\FF_\e\) and the \(\BdR^+\)\=/module \(\FF_\dR^+\) are free.

The Galois group \(\GG_K\) acts continuously on \(\XFF\), so there are Galois equivariant coherent sheaves over \(\XFF\).
Since the closed point \(\infty\) and its complement \(\XFF \setminus \{\infty\}\) are stable under the action of \(\GG_K\), Theorem~\ref{theo:BeauvilleLaszlo} yields a description of Galois-equivariant coherent sheaves over \(\XFF\).
Precisely, let \(L\) be an algebraic extension of \(K\), and let \(\MM(\GG_L)\) be the fibre product of the categories
\[
	\begin{tikzcd}
									& \Rep_{\BdR^+}(\GG_L) \ar{d} \\
		\Rep_{\Be}(\GG_L) \ar{r}	& \Rep_{\BdR}(\GG_L),
	\end{tikzcd}
\]
where the arrows are the functors of extension of scalars.
Then \(\MM(\GG_L)\) is a subcategory of \(\MM\) whose objects are triple \((\FF_\e, \FF_\dR^+, \iota_{\FF})\) with
\begin{itemize}
	\item \(\FF_\e\) a \(\Be\)\=/representation of \(\GG_L\),
	\item \(\FF_\dR^+\) a \(\BdR^+\)\=/representation of \(\GG_L\),
	\item the isomorphism \(\iota_{\FF}\) is \(\GG_L\)\=/equivariant,
\end{itemize}
and whose morphisms are maps \((f_\e, f_\dR^+)\) such that both \(f_\e\) and \(f_\dR^+\) are \(\GG_L\)\=/equivariant.
Let \(\Coh(\GG_L)\) be the subcategory of \(\Coh\) of \(\GG_L\)\=/equivariant coherent sheaves over \(\XFF\).
Then, the equivalence between \(\Coh\) and \(\MM\) of Theorem~\ref{theo:BeauvilleLaszlo} induces an equivalence between the categories \(\Coh(\GG_L)\) and \(\MM(\GG_L)\).

We denote by \(\Bun(\GG_K)\) the full subcategory of \(\Coh(\GG_K)\) of \(\GG_K\)\=/equivariant vector bundles over \(\XFF\) whose objects are \(\GG_K\)\=/equivariant coherent sheaves over \(\XFF\) with underlying torsion free coherent \(\OO_{\XFF}\)\=/modules.
Fontaine proved~\cite[Proposition~3.1]{Fontaine20} that the \(\Be\)\=/module underlying any \(\Be\)\=/representation of \(\GG_K\) is free.
Therefore, a \(\GG_K\)\=/equivariant coherent sheaf \(\FF\) is a vector bundle if and only if the \(\BdR^+\)\=/module \(\FF_\dR^+\) is free.
Let \(\Rep_{\BdR^+}^\fr(\GG_K)\) be the full subcategory of \(\Rep_{\BdR^+}(\GG_K)\) of \(\BdR^+\)\=/representations of \(\GG_K\) whose underlying \(\BdR^+\)\=/module is free.
Let \(\MM^\fr(\GG_K)\) be fibre product of the categories
\[
	\begin{tikzcd}
									& \Rep_{\BdR^+}^\fr(\GG_K) \ar{d}\\
		\Rep_{\Be}(\GG_K) \ar{r}	& \Rep_{\BdR}(\GG_K).
	\end{tikzcd}
\]
Then, under the equivalence of Theorem~\ref{theo:BeauvilleLaszlo}, the category \(\Bun(\GG_K)\) is equivalent to \(\MM^\fr(\GG_K)\).

\begin{rema}
	The category \(\MM^\fr(\GG_K)\) is the category of \emph{\(B\)\=/pairs} introduced by Berger~\cite{Berger08}.
\end{rema}

From now on, we identify the categories \(\Coh\) and \(\MM\).
We shall write \(\FF = (\FF_\e,\FF_\dR^+,\iota_{\FF})\) a coherent sheaf over \(\XFF\).
We will omit the map \(\iota_{\FF}\) when there is no ambiguity.
If \(\FF_\dR^+\) is a \(\BdR^+\)\=/module, then we set
\[
	\FF_\dR=\BdR \otimes_{\BdR^+} \FF_\dR^+.
\]

\subsection{The Harder-Narasimhan filtration} \label{subsec:HN}
We recall the theory of Harder-Narasimhan slope filtration for coherent sheaves over \(\XFF\).
The Abelian category \(\Coh\) is endowed with two additive functions \emph{degree} and \emph{rank}
\[
	\begin{split}
		\deg : 	\Coh & \ra \Z,\\
		\rk  :	\Coh & \ra \N.\\
	\end{split}
\]
The rank of a coherent sheaf \(\FF = (\FF_\e, \FF_\dR^+, \iota_{\FF})\) is its rank as an \(\OO_{\XFF}\)\=/module, equivalently, it is the rank of the \(\Be\)\=/module \(\FF_\e\), or also, the rank of the \(\BdR^+\)\=/module \(\FF_\dR^+\).
The additive function degree is characterized by the following.
\begin{itemize}
	\item The degree of an invertible sheaf \(\LL\) is the degree of the divisor associated with any non-zero section of \(\LL\). The degree of \(\LL\) is then well-defined since \(\XFF\) is complete.
	\item If \(\EE\) is a vector bundle of rank \(r\), then the degree of \(\EE\) is the degree of the determinant line bundle \(\bigwedge^r \EE\) of \(\EE\).
	\item If \(\FF\) is a torsion coherent sheaf, then
\[
	\deg(\FF) = \sum_{\text{closed point } x \in \XFF} \length_{\OO_{\XFF,x}}( \FF_x ).
\]
\end{itemize}

The \emph{slope} of a coherent sheaf \(\FF\) is
\[
	\mu(\FF) = \left\{ \begin{array}{ll}
		\frac{\deg(\FF)}{\rk(\FF)} & \text{if } \FF \text{ is non-torsion},\\
		+\infty & \text{if } \FF \text{ is torsion}.
	\end{array}\right.
\]
A coherent sheaf \(\FF\) over \(\XFF\) is \emph{semi-stable} if it is non-trivial, and
\[
	\mu(\FF^\prime) \leq \mu(\FF)
\]
for every non-trivial coherent subsheaf \(\FF^\prime \subset \FF\).

\begin{theo} \label{theo:HN}
	Let \(\FF\) be a coherent sheaf over \(\XFF\).
	There exists a unique increasing filtration of \(\FF\) by coherent subsheaves
	\[
		0 = \FF_0 \subset \FF_1 \subset \cdots \subset \FF_{n-1} \subset \FF_n = \FF
	\]
	such that \(\FF_i/\FF_{i-1}\) is semi-stable for each \(i \in \{1,\dots,n\}\), and
	\[
		\mu(\FF_1/\FF_0) > \mu(\FF_2/\FF_1) > \cdots > \mu(\FF_{n-1}/\FF_{n-2}) > \mu(\FF_n/\FF_{n-1}).
	\]
\end{theo}

If \(\FF\) is a coherent sheaf over \(\XFF\), then the filtration \((\FF_i)_{0 \leq i \leq n}\) of \(\FF\) from Theorem~\ref{theo:HN} is the \emph{Harder\-/Narasimhan filtration} of \(\FF\), and the slopes \((\mu(\FF_i/\FF_{i-1}))_{1 \leq i \leq n}\) are the \emph{Harder\-/Narasimhan slopes} of \(\FF\).

\subsection{Cohomology of coherent sheaves} \label{subsec:cohom}
Let \(\FF=(\FF_\e, \FF_\dR^+, \iota_{\FF})\) be a coherent sheaf over \(\XFF\).
The description à la Beauville-Laszlo of \(\FF\) yields the following description of its cohomology~\cite[Proposition~5.3.3]{FarguesFontaine18}.
\begin{theo} \label{theo:cohomCoh}
	The complex \(\RGamma(\XFF,\FF)\) is canonically and functorially isomorphic to the complex in degree \(0\) and \(1\)
	\[
		\left[\begin{aligned}
			\FF_\e \oplus \FF_\dR^+ & \ra \FF_\dR\\
			(x,y) & \mapsto \iota_{\FF}(x) - y
		\end{aligned}\right].
	\]
\end{theo}

\begin{rema}
	The \(\Qp\)\=/vector space of global sections \(\Gamma(\XFF, \FF)\) is naturally equipped with the subspace topology from \(\FF_\e \oplus \FF_\dR^+\) by Theorem~\ref{theo:cohomCoh} (making \(\Gamma(\XFF, \FF)\) into a \(p\)\=/adic Banach space).
	In particular, if \(L\) is an algebraic extension of \(K\) and \(\FF\) is \(\GG_L\)\=/equivariant, then \(\Gamma(\XFF,\FF)\) is a topological \(\GG_L\)\=/module.
\end{rema}

Moreover, we have the following relation between the cohomology of \(\FF\) and its Harder\-/Narasimhan filtration~\cite[Proposition~8.2.3]{FarguesFontaine18}.
\begin{samepage}
\begin{theo}[Fargues \& Fontaine] \label{theo:cohomHN}
	\
	\begin{enumerate}
		\item The group \(\H^0(\XFF,\FF)\) vanishes if and only if the Harder\-/Narasimhan slopes of \(\FF\) are strictly less than \(0\).
		\item The group \(\H^1(\XFF,\FF)\) vanishes if and only if the Harder\-/Narasimhan slopes of \(\FF\) are greater than or equal to \(0\).
	\end{enumerate}
\end{theo}
\end{samepage}

\begin{rema}
	The exact sequence induced by Theorem~\ref{theo:cohomCoh} for the structure sheaf \(\OO_{\XFF}\) is none other than the fundamental exact sequence~\eqref{eq:funda}
	\[
		\begin{tikzcd}
			0 \ar{r} & \H^0(\XFF, \OO_{\XFF}) \ar{r} \ar[equal]{d} & (\OO_{\XFF})_\e \ar{r} \ar[equal]{d} & \frac{(\OO_{\XFF})_\dR}{(\OO_{\XFF})_\dR^+} \ar{r} \ar[equal]{d} & \H^1(\XFF,\OO_{\XFF}) \ar[equal]{d} \\
			0 \ar{r} & \Qp \ar{r} & \Be \ar{r} & \BdR/\BdR^+ \ar{r} & 0.
		\end{tikzcd}
	\]
	The vanishing of \(\H^1(\XFF,\OO_{\XFF})\) is consistent with Theorem~\ref{theo:cohomHN} since the structure sheaf \(\OO_{\XFF}=(\Be, \BdR^+)\) is semi-stable of slope \(0\).
\end{rema}

\subsection{Classification of vector bundles} \label{subsec:classification}
Fargues and Fontaine have classified vector bundles over \(\XFF\) \cite[Théorème~8.2.10]{FarguesFontaine18}.
In particular, they proved that the Harder\-/Narasimhan filtration of a coherent sheaf is split (non-canonically).

Let \(L\) be an algebraic extension of \(K\).
Let \(\FF\) be a \(\GG_L\)\=/equivariant coherent sheaf over \(\XFF\).
The unicity of the Harder\-/Narasimhan filtration of \(\FF\) implies that it is preserved under the action of \(\GG_L\), \ie each piece of the Harder\-/Narasimhan filtration of \(\FF\) is a \(\GG_L\)\=/equivariant coherent subsheaf of \(\FF\).
However, there is in general no \(\GG_L\)\=/equivariant splitting of the Harder\-/Narasimhan filtration.

We recall~\cite{Scholze12} that a complete non-Archimedean field \(F\) of residue characteristic \(p\) is a \emph{perfectoid field} if its valuation group is non-discrete and the \(p\)\=/th power Frobenius map on \(\OO_F /(p)\) is surjective.
If \(\h{L}\) is a perfectoid field, Fargues and Fontaine proved that the Harder\-/Narasimhan filtration of \(\FF\) is split \(\GG_L\)\=/equivariant~\cite[Théorème~9.4.1]{FarguesFontaine18}.

As a by-product of this result, they obtained the following~\cite[Remarque~9.4.2]{FarguesFontaine18} which will be crucial for us.
\begin{prop}[Fargues \& Fontaine] \label{prop:FFsplit}
	If \(\h{L}\) is a perfectoid field and the Harder\-/Narasimhan slopes of \(\FF\) are strictly greater than \(0\), then
	\[
		\H^1(L,\Gamma(\XFF,\FF))=0.
	\]
\end{prop}
\begin{rema} \label{rema:FST}
	The case of a torsion coherent sheaf in Proposition~\ref{prop:FFsplit} is equivalent to the following result from Fontaine-Sen-Tate theory~\cite[Proposition~7.1.1]{FarguesFontaine18}.
	Let \(\HH_\dR^+\) be \(\BdR^+\)\=/representation of \(\GG_L\) whose underlying \(\BdR^+\)\=/module is of finite length.
	If \(\h{L}\) is perfectoid, then \(\H^1(L,\HH_\dR^+)\) is trivial.
\end{rema}

Another consequence~\cite[Théorème~10.1.7]{FarguesFontaine18} of the classification of vector bundles over the Fargues-Fontaine curve concerns \(p\)\=/adic representations.
Let \(\Bun^0(\GG_K)\) be the full subcategory of \(\Bun(\GG_K)\) of \(\GG_K\)\=/equivariant vector bundles semi-stable of slope \(0\).
There is a functor from \(\Rep_{\Qp}(\GG_K)\) to \(\Bun^0(\GG_K)\) which associates with a \(p\)\=/adic representation \(V\) of \(\GG_K\) the vector bundle
\[
	\EE(V) = \OO_{\XFF} \otimes_{\Qp} V = (\Be \otimes_{\Qp} V, \BdR^+ \otimes_{\Qp} V).
\]
\begin{theo}[Fargues \& Fontaine] \label{theo:RepBun0}
	The functor
	\[
		\begin{split}
			\Rep_{\Qp}(\GG_K) & \ra \Bun^0(\GG_K)\\
			V & \mapsto \EE(V)
		\end{split}
	\]
	is an equivalence of categories, of which the functor
	\[
		\begin{split}
			\Bun^0(\GG_K) & \ra \Rep_{\Qp}(\GG_K)\\
			\EE & \mapsto \Gamma(\XFF,\EE)
		\end{split}
	\]
	is a quasi-inverse.
\end{theo}

\subsection{De Rham vector bundles} \label{subsec:dR}
We recall the definition of the functor \(\DdR\) from the category of \(\GG_K\)\=/equivariant vector bundles to the category of filtered \(K\)\=/vector spaces.
We first review the definition of a filtered \(K\)\=/vector space.

A \emph{filtered \(K\)\=/vector space} is a finite dimensional \(K\)\=/vector space \(D\) endowed with a filtration \(\Fil D = \{\Fil^n D\}_{n \in \Z}\) by sub\=/\(K\)\=/vector spaces which is
\begin{itemize}
	\item decreasing, \ie \(\Fil^{n+1} D \subset \Fil^n D\), for all \(n \in \Z\),
	\item exhaustive, \ie \(\bigcup_{n \in \Z} \Fil^n D = D\),
	\item and separated, \ie \(\bigcap_{n \in \Z} \Fil^n D = \{0\}\).
\end{itemize}
The \emph{weights} of a filtered \(K\)\=/vector space \((D,\Fil D)\) are the integers \(n\) such that \(\Fil^{-n} D /\Fil^{-n+1} D \neq 0\).
The \emph{multiplicity} of an integer \(n\) as a weight of \((D,\Fil D)\) is the dimension \(\dim_K (\Fil^{-n} D /\Fil^{-n+1} D)\).

A morphism of filtered \(K\)\=/vector spaces
\[
	(D, \Fil D) \ra (C, \Fil C)
\]
is a \(K\)\=/linear application \(f: D \ra C\) compatible with the filtrations, \ie \(f(\Fil^n D) \subset \Fil^n C\), for all \(n \in \Z\).
Let \(\Fil_K\) be the category of filtered \(K\)\=/vector spaces.
The category \(\Fil_K\) is not Abelian, nonetheless a sequence
\[
	0 \ra (D^\prime , \Fil D^\prime) \ra (D, \Fil D) \ra (D^\dprime , \Fil D^\dprime) \ra 0
\]
is said to be \emph{exact} if for all \(n \in \Z\) it induces a short exact sequence of \(K\)\=/vector spaces
\[
	0 \ra \Fil^n D^\prime \ra \Fil^n D \ra \Fil^n D^\dprime \ra 0.
\]

\begin{rema} \label{rema:DdR}
	If
	\[
		0 \ra (D^\prime , \Fil D^\prime) \ra (D, \Fil D) \ra (D^\dprime , \Fil D^\dprime) \ra 0
	\]
	is a short exact sequence, then the union of the sets of weights of \((D^\prime,\Fil D^\prime)\) and \((D^\dprime, \Fil D^\dprime)\) (counted with multiplicity) is the set of weights of \((D,\Fil D)\).
	In particular, we have the following criterion which we shall use repeatedly later on: the weights of \((D, \Fil D)\) are less than or equal to \(0\) if and only if the weights of both \((D^\prime, \Fil D^\prime)\) and \((D^\dprime, \Fil D^\dprime)\) are less than or equal to \(0\).
\end{rema}

We proceed to the construction of the functor \(\DdR\).
First, a \(K\)\=/vector space is associated with a \(\BdR\)\=/representation of \(\GG_K\).
Let \(\Vec_K\) be the category of finite dimensional \(K\)\=/vector spaces.
We recall that \(\BdR^{\GG_K}=K\).
Thus, there exists a functor
\begin{equation} \label{eq:BdRVec}
	\begin{split}
		\Rep_{\BdR}(\GG_K) & \ra \Vec_K\\
		\EE_\dR & \mapsto \EE_\dR^{\GG_K},
	\end{split}
\end{equation}
which is left adjoint to the functor
\begin{equation} \label{eq:VecBdR}
	\begin{split}
		\Vec_K & \ra \Rep_{\BdR}(\GG_K)\\
		D & \mapsto \BdR \otimes_K D.
	\end{split}
\end{equation}
A \(\BdR\)\=/representation \(\EE_\dR\) of \(\GG_K\) is \emph{flat} if the injection
\[
	\BdR \otimes_K \EE_\dR^{\GG_K} \ra \EE_\dR
\]
is an isomorphism, equivalently, we have \(\dim_K \EE_\dR^{\GG_K} \leq \dim_{\BdR} \EE_\dR\) and \(\EE_\dR\) is flat if \(\dim_K \EE_\dR^{\GG_K} = \dim_{\BdR} \EE_\dR\).
We denote by \(\Rep_{\BdR}^\fl(\GG_K)\) the full subcategory of \(\Rep_{\BdR}(\GG_K)\) of flat \(\BdR\)\=/representations of \(\GG_K\).
Then, the functor~\eqref{eq:BdRVec} induces an equivalence of categories
\begin{equation} \label{eq:flatVec}
	\Rep_{\BdR}^\fl(\GG_K) \riso \Vec_K,
\end{equation}
of which the functor~\eqref{eq:VecBdR} is a quasi-inverse.

If \(\EE_\dR^+\) is a free \(\BdR^+\)\=/representation of \(\GG_K\), then \(\EE_\dR^+\) is a \(\BdR^+\)\=/lattice in \(\EE_\dR\) stable under the action of \(\GG_K\).
Since the filtration \(\{\Fil^n \BdR = \BdR^+t^n\}_{n \in \Z}\) on \(\BdR\) is exhaustive, separated, \(\GG_K\)\=/stable and satisfies~\eqref{eq:FilGalois}, the \(K\)\=/vector space \(\EE_\dR^{\GG_K}\) equipped with the induced filtration \(\{(t^n \EE_\dR^+)^{\GG_K}\}_{n \in \Z}\) is a filtered \(K\)\=/vector space.
Thus, there is a functor
\begin{equation} \label{eq:freeFil}
	\begin{split}
		\Rep_{\BdR^+}^\fr(\GG_K) & \ra \Fil_K\\
		\EE_\dR^+ & \mapsto \left(\EE_\dR^{\GG_K}, \{(t^n \EE_\dR^+)^{\GG_K}\}_{n \in \Z}\right).
	\end{split}
\end{equation}
Reciprocally, if \((D,\Fil D)\) is a filtered \(K\)\=/vector space, then
\[
	\left(\sum_{n \in \Z} \Fil^n \BdR \otimes_K \Fil^{-n} D\right) \subset \BdR \otimes_K D
\]
is a \(\BdR^+\)\=/lattice in \(\BdR \otimes_K D\) stable under the action of \(\GG_K\), and there is a functor
\begin{equation} \label{eq:Filfree}
	\begin{split}
		\Fil_K & \ra \Rep_{\BdR^+}^\fr(\GG_K)\\
		(D, \Fil D) & \mapsto \sum_{n \in \Z} \Fil^n \BdR \otimes_K \Fil^{-n} D,
	\end{split}
\end{equation}
which is right adjoint to the functor~\eqref{eq:freeFil}.

A free \(\BdR^+\)\=/representation \(\EE_\dR^+\) of \(\GG_K\) is \emph{generically flat} if the \(\BdR\)\=/representation \(\EE_\dR\) is flat.
Let \(\Rep_{\BdR^+}^\fl(\GG_K)\) be the full subcategory of \(\Rep_{\BdR^+}^\fr(\GG_K)\) of generically flat \(\BdR^+\)\=/representation of \(\GG_K\).
Fargues and Fontaine have classified the generically flat \(\BdR^+\)\=/representations of \(\GG_K\) \cite[Théorème~10.4.4]{FarguesFontaine18}.
\begin{theo}[Fargues \& Fontaine] \label{theo:genflatFil}
	The functor~\eqref{eq:freeFil} induces an equivalence of categories
	\[
		\Rep_{\BdR^+}^\fl(\GG_K) \riso \Fil_K,
	\]
	of which the functor~\eqref{eq:Filfree} is a quasi-inverse.
\end{theo}

\begin{rema} \label{rema:diagcommutFilRep}
	There is a commutative diagram
	\[
		\begin{tikzcd}
			\Rep_{\BdR^+}^\fl(\GG_K) \ar{r} \ar{d}[sloped]{\sim} & \Rep_{\BdR}^\fl(\GG_K) \ar{d}[sloped]{\sim} \\
			\Fil_K \ar{r} & \Vec_K,
		\end{tikzcd}
	\]
	where the vertical equivalences are respectively~\eqref{eq:flatVec} and Theorem~\ref{theo:genflatFil}, the top functor is the functor of extension of scalars
	\[
		\begin{split}
			\Rep_{\BdR^+}^\fl(\GG_K) & \ra \Rep_{\BdR}^\fl(\GG_K)\\
			\EE_\dR^+ & \mapsto \EE_\dR,
		\end{split}
	\]
	and the bottom functor is the forgetful functor
	\[
		\begin{split}
			\Fil_K & \ra \Vec_K\\
			(D, \Fil D) & \mapsto D.
		\end{split}
	\]
\end{rema}

The aforementioned functor \(\DdR\) is defined as the composition of functors
\[
	\begin{array}{r@{\ }c@{\ }c@{\ }c@{\ }l}
		\DdR: \Bun(\GG_K) & \ra & \Rep_{\BdR^+}^\fr(\GG_K) & \ra & \Fil_K\\
		\EE & \mapsto & \EE_\dR^+ & \mapsto & (\DdR(\EE),\Fil \DdR(\EE)),
	\end{array}
\]
where we set
\[
	(\DdR(\EE), \Fil \DdR(\EE))=\left(\EE_\dR^{\GG_K}, \{(t^n \EE_\dR^+)^{\GG_K}\}_{n \in \Z}\right).
\]
The filtration \(\Fil \DdR(\EE)\) of \(\DdR(\EE)\) is called its \emph{Hodge-Tate filtration}.
We will omit it and simply denote \((\DdR(\EE), \Fil \DdR(\EE))\) by \(\DdR(\EE)\).

A \(\GG_K\)\=/equivariant vector bundle \(\EE=(\EE_\e , \EE_\dR^+, \iota_{\EE})\) over \(\XFF\) is \emph{de Rham} if \(\EE_\dR^+\) is generically flat, equivalently, we have \(\dim_K \DdR(\EE) \leq \rk \EE\) and \(\EE\) is de Rham if \(\dim_K \DdR(\EE) = \rk \EE\).
The \emph{Hodge-Tate weights} of a de Rham vector bundle \(\EE\) are the weights of \(\DdR(\EE)\).
We denote by \(\Bun(\GG_K)_\dR\) the full subcategory of \(\Bun(\GG_K)\) of \(\GG_K\)\=/equivariant de Rham vector bundles over \(\XFF\).
Under the equivalence of Theorem~\ref{theo:BeauvilleLaszlo}, the category \(\Bun(\GG_K)_\dR\) is equivalent to the fibre product of
\[
	\begin{tikzcd}
		& \Rep_{\BdR^+}^\fl(\GG_K) \ar{d}\\
		\Rep_{\Be}(\GG_K) \ar{r} & \Rep_{\BdR}(\GG_K).
	\end{tikzcd}
\]

\begin{rema} \label{rema:DdRRep}
	The composition of functors
	\[
		\begin{array}{r@{\ }c@{\ }c@{\ }c@{\ }l}
			\Rep_{\Qp}(\GG_K) & \ra & \Bun(\GG_K) & \ra & \Fil_K\\
			V & \mapsto & \EE(V) & \mapsto & \DdR(\EE(V))
		\end{array}
	\]
	is the usual \(\DdR\) functor~\cite{Fontaine94II}.
	In particular, a \(p\)\=/adic representation \(V\) is de Rham if and only if the vector bundle \(\EE(V)\) is de Rham.
	We shall write \(\DdR(V)\) instead of \(\DdR(\EE(V))\).
\end{rema}

The following Proposition, which is essentially contained in~\cite{FarguesFontaine18}, is a generalisation of the analogous result for \(p\)-adic representations~\cite[\S 1.5]{Fontaine94II}, and is proved similarly.
\begin{prop} \label{prop:DdR}\
	Let
	\[
		0 \ra \EE^\prime \ra \EE \ra \EE^\dprime \ra 0
	\]
	be a short exact sequence in \(\Bun(\GG_K)\).
	If \(\EE\) is de Rham, then \(\EE^\prime\) and \(\EE^\dprime\) are de Rham, and the associated sequence in \(\Fil_K\)
	\[
		0 \ra \DdR(\EE^\prime) \ra \DdR(\EE) \ra \DdR(\EE^\dprime) \ra 0
	\]
	is exact.
\end{prop}
\begin{proof}
	The short exact sequence
	\begin{equation} \label{eq:bun}
		0 \ra \EE^\prime \ra \EE \ra \EE^\dprime \ra 0
	\end{equation}
	implies the equality
	\begin{equation} \label{eq:bunrank}
		\rk \EE = \rk \EE^\prime + \rk \EE^\dprime.
	\end{equation}
	Additionally, the short exact sequence~\eqref{eq:bun} induces a short exact sequence in \(\Rep_{\BdR}(\GG_K)\)
	\begin{equation} \label{eq:BdR}
		0 \ra \EE^\prime_\dR \ra \EE_\dR \ra \EE^\dprime_\dR \ra 0,
	\end{equation}
	and taking \(\GG_K\)-invariants of this last exact sequence~\eqref{eq:BdR} yields an exact sequence of \(K\)-vector spaces
	\begin{equation} \label{eq:DdRham}
				0 \ra \DdR(\EE^\prime) \ra \DdR(\EE) \ra \DdR(\EE^\dprime).
	\end{equation}
	Hence, by~\eqref{eq:DdRham} , we have
	\begin{equation} \label{eq:dimDdR}
		\dim_K \DdR(\EE) \leq \dim_K \DdR(\EE^\prime) + \dim_K \DdR(\EE).
	\end{equation}
	Moreover, we have
	\begin{equation} \label{eq:dimineq}
		\begin{split}
			\dim_K \DdR(\EE^\prime) & \leq \rk \EE^\prime\\
			\dim_K \DdR(\EE^\dprime) & \leq \rk \EE^\dprime,
		\end{split}
	\end{equation}
	and, since \(\EE\) is de Rham, we have
	\begin{equation} \label{eq:dimdR}
		\dim_K \DdR(\EE) = \rk \EE.
	\end{equation}
	The combination of the equations~\eqref{eq:bunrank}, \eqref{eq:dimDdR}, \eqref{eq:dimineq} and~\eqref{eq:dimdR} implies
	\[
		\begin{split}
			\dim_K \DdR(\EE^\prime) & = \rk \EE^\prime\\
			\dim_K \DdR(\EE^\dprime) & = \rk \EE^\dprime,
		\end{split}
	\]
	and therefore, \(\EE^\prime\) and \(\EE^\dprime\) are de Rham.

	Furthermore, since all the terms of the short exact sequence~\eqref{eq:bun} are de Rham, it induces, by definition, a short exact sequence of generically flat \(\BdR^+\)-representations of \(\GG_K\)
	\begin{equation} \label{eq:genflatbun}
		0 \ra (\EE^\prime)_\dR^+ \ra \EE_\dR^+ \ra (\EE^\dprime)_\dR^+ \ra 0,
	\end{equation}
	and thus, by Theorem~\ref{theo:genflatFil}, the sequence of filtered \(K\)-vector spaces associated with~\eqref{eq:genflatbun} is short exact.
\end{proof}

\begin{rema}
	Let \(n \in \Z\).
	With the above convention for weights, the Hodge-Tate weight of the representation \(\Qp(n)\) is \(n\).
\end{rema}

\section{Truncation of the Hodge-Tate filtration} \label{sec:trunc}
In this section, we define and study the \emph{modification by truncation} of a de Rham vector bundle.
This modification is induced by truncation of the Hodge-Tate filtration of the associated filtered vector space.

\subsection{Modification of filtered vector spaces} \label{subsec:truncFil}
Let \(\Fil_K^{\leq 0}\) be the full subcategory of \(\Fil_K\) of filtered \(K\)\=/vector spaces whose weights are less than or equal to \(0\).
We construct a left adjoint to the forgetful functor from \(\Fil_K^{\leq 0}\) to \(\Fil_K\).

\begin{defi}
	Let \((D, \Fil D)\) be a filtered \(K\)\=/vector space.
	The \emph{modification by truncation} of \((D,\Fil D)\) is the filtered \(K\)\=/vector space \((D,\Fil_+ D)\) where
	\[
		\Fil_+^n D = \left\{ \begin{array}{ll}
			D & \text{if } n \leq 0,\\
			\Fil^n D & \text{if } n > 0.
		\end{array} \right.
	\]
\end{defi}

The modification by truncation \((D,\Fil_+ D)\) of a filtered \(K\)\=/vector space \((D, \Fil D)\) is an object of \(\Fil_K^{\leq 0}\).
Moreover, the association \((D, \Fil D) \mapsto (D, \Fil_+ D)\) is functorial: if
\[
	f: (D, \Fil D) \ra (C,\Fil C)
\]
is a morphism of filtered \(K\)\=/vector spaces, that is a \(K\)\=/linear map \(f: D \ra C\) compatible with the filtrations, then the map \(f\) is compatible with the truncated filtrations and induces a morphism between the modifications by truncation
\[
	f: (D, \Fil_+ D) \ra (C,\Fil_+ C).
\]

\begin{defi}
	The functor induced by the modification by truncation is called the \emph{truncation functor} and denoted by
	\[
		\begin{split}
			\trunc_{\Fil} : \Fil_K & \ra \Fil_K^{\leq 0}\\
			(D,\Fil D) & \mapsto (D, \Fil_+ D).
		\end{split}
	\]
\end{defi}

\begin{rema} \label{rema:diagtruncFil}
	The truncation functor fits into the commutative diagram
	\[
		\begin{tikzcd}
			\Fil_K \ar{r}{\trunc_{\Fil}} \ar{d} & \Fil_K^{\leq 0} \ar{d}\\
			\Vec_K \ar{r}{\id} & \Vec_K,
		\end{tikzcd}
	\]
	where the vertical arrows are the forgetful functor, and the bottom arrow is the identity functor.
\end{rema}

\begin{prop} \label{prop:truncFil}
	The truncation functor \(\trunc_{\Fil}\) is exact and left adjoint to the forgetful functor.
\end{prop}
\begin{proof}
	The exactness follows directly from the definition.
	We prove the adjunction property.
	If \((C, \Fil C)\) is an object of \(\Fil_K^{\leq 0}\), then by definition we have
	\[
		(C, \Fil_+ C) = (C, \Fil C).
	\]
	Therefore, the truncation functor provides a map functorial in \((D,\Fil D) \in \Obj(\Fil_K)\) and \((C,\Fil C) \in \Obj(\Fil_K^{\leq 0})\):
	\begin{equation} \label{eq:truncadj}
		\begin{tikzcd}
			\Hom((D,\Fil D), (C,\Fil C)) \ar{d}{\trunc_{\Fil}} \\
			\Hom((D,\Fil_+ D), (C,\Fil_+ C)) \ar[equal]{d}\\
			\Hom((D,\Fil_+ D), (C,\Fil C)).
		\end{tikzcd}
	\end{equation}
	Reciprocally, if \(f \in \Hom((D,\Fil_+ D), (C,\Fil C))\), then the underlying \(K\)\=/linear map \(f : D \ra C\) satisfies
	\begin{equation} \label{eq:Filpos}
		f(\Fil^n D) \subset \Fil^n C, \text{ if } n > 0.
	\end{equation}
	Moreover, since \((C, \Fil^i C)\) is an object of \(\Fil_K^{\leq 0}\), the map \(f\) also satisfies
	\begin{equation} \label{eq:Filneg}
		f(\Fil^n D) \subset C=\Fil^n C, \text{ if } n \leq 0.
	\end{equation}
	Therefore, the combination of the equations~\eqref{eq:Filpos} and~\eqref{eq:Filneg} implies that the map \(f\) induces a morphism \((D,\Fil D) \ra (C,\Fil C)\).
	Thus, the map~\eqref{eq:truncadj} is bijective and functorial in \((D,\Fil D)\) and \((C,\Fil C)\), so the truncation functor is left adjoint to the forgetful functor.
\end{proof}

If \((D, \Fil D)\) is a filtered \(K\)\=/vector space, then the identity map on \(D\) induced a morphism of filtered \(K\)\=/vector spaces
\[
	\eta_{(D, \Fil D)} : (D, \Fil D) \ra (D, \Fil_+ D).
\]
Moreover, the map \(\eta_{(D, \Fil D)}\) is send to the identity map of \((D,\Fil_+ D)\) by the bijection~\eqref{eq:truncadj}:
\[
	\begin{split}
		\trunc_{\Fil} : \Hom((D,\Fil D), (D,\Fil_+ D)) & \riso \Hom((D,\Fil_+ D), (D,\Fil_+ D))\\
		\eta_{(D, \Fil D)} & \mapsto \id_{(D,\Fil_+ D)}.
	\end{split}
\]
Thus, we have:
\begin{coro} \label{coro:universalmorphFil}
	Let \((D, \Fil D)\) be a filtered \(K\)\=/vector space.
	The map \(\eta_{(D, \Fil D)}\) is the universal morphism from \((D, \Fil D)\) to the forgetful functor from \(\Fil_K^{\leq 0}\) to \(\Fil_K\).
\end{coro}

Via the equivalence of categories between \(\Fil_K\) and \(\Rep_{\BdR^+}^\fl(\GG_K)\) (Theorem~\ref{theo:genflatFil}), the modification by truncation of filtered \(K\)\=/vector spaces defines a modification of generically flat \(\BdR^+\)\=/representations of \(\GG_K\) which we now compute.
\begin{lemm} \label{lemm:fillattice}
	Let \((D,\Fil D)\) be a filtered \(K\)\=/vector space.
	Then
	\[
		\sum_{n \in \Z} \Fil^n \BdR \otimes_K \Fil_+^{-n} D = \left(\sum_{n \in \Z} \Fil^n \BdR \otimes_K \Fil^{-n} D\right) + \BdR^+ \otimes_K D.
	\]
\end{lemm}
\begin{proof}
	By definition of \(\Fil_+ D\), we have
	\begin{equation}
		\begin{split} \label{eq:Fil}
			& \sum_{n \in \Z} \Fil^n \BdR \otimes_K \Fil_+^{-n} D\\
			= & \sum_{n < 0} \Fil^n \BdR \otimes_K \Fil^{-n} D + \sum_{n \geq 0} \Fil^n \BdR \otimes_K D.
		\end{split}
	\end{equation}
	Since \(\Fil^n \BdR \otimes_K \Fil^{-n} D\) is contained in \(\Fil^n \BdR \otimes_K D\) for all \(n \in \Z\), the equation~\eqref{eq:Fil} gives
	\begin{equation} \label{eq:Fil2}
	\begin{split}
		& \sum_{n \in \Z} \Fil^n \BdR \otimes_K \Fil^{-n}_+ D \\
		= & \sum_{n \in \Z} \Fil^i \BdR \otimes_K \Fil^{-n} D + \sum_{n \geq 0} \Fil^n \BdR \otimes_K D.
	\end{split}
	\end{equation}
	Moreover, we have
	\begin{equation} \label{eq:Fil3}
			\sum_{n \geq 0} \Fil^n \BdR \otimes_K D = \BdR^+ \otimes_K D.
	\end{equation}
	Therefore, the combination of the equations~\eqref{eq:Fil2} and \eqref{eq:Fil3} yields
	\[
		\sum_{n \in \Z} \Fil^n \BdR \otimes_K \Fil^{-n}_+ D = \left(\sum_{n \in \Z} \Fil^n \BdR \otimes_K \Fil^{-n} D\right) + \BdR^+ \otimes_K D.
	\]
\end{proof}

Let \(\Rep_{\BdR^+}^\fl(\GG_K)^{\leq 0}\) be the full subcategory of \(\Rep_{\BdR^+}^\fl(\GG_K)\) equivalent to \(\Fil_K^{\leq 0}\) under the equivalence between \(\Rep_{\BdR^+}^\fl(\GG_K)\) and \(\Fil_K\) of Theorem~\ref{theo:genflatFil}.
By Lemma~\ref{lemm:fillattice}, the composition of functors
\begin{equation} \label{eq:truncRepDef}
	\Rep_{\BdR^+}^\fl(\GG_K) \riso \Fil_K \xra{\trunc_{\Fil}} \Fil_K^{\leq 0} \riso \Rep_{\BdR^+}^\fl(\GG_K)^{\leq 0}
\end{equation}
associates with a generically flat \(\BdR^+\)\=/representation \(\EE_\dR^+\) of \(\GG_K\) the generically flat \(\BdR^+\)\=/representation
\[
	\EE_\dR^+ + \BdR^+ \otimes_K (\EE_\dR)^{\GG_K}.
\]

\begin{defi}
	The functor~\eqref{eq:truncRepDef} is also called the \emph{truncation functor}, and is denoted by
	\[
		\begin{split}
			\trunc_{\dR}: \Rep_{\BdR^+}^\fl(\GG_K) 	& \ra \Rep_{\BdR^+}^\fl(\GG_K)^{\leq 0}\\
								\EE_\dR^+ 		& \mapsto \EE_\dR^+ + \BdR^+ \otimes_K (\EE_\dR)^{\GG_K}.
		\end{split}
	\]
\end{defi}

The next corollary follows immediately from the equivalence of Theorem~\ref{theo:genflatFil} and Remark~\ref{rema:diagcommutFilRep}, and the properties of the truncation functor \(\trunc_{\Fil}\): Remark~\ref{rema:diagtruncFil} and Proposition~\ref{prop:truncFil}.
\begin{coro} \label{coro:truncRep}
	The truncation functor \(\trunc_{\dR}\) is exact and left adjoint to the forgetful functor.
	Furthermore, the diagram
	\[
		\begin{tikzcd}
			\Rep_{\BdR^+}^\fl(\GG_K) \ar{r}{\trunc_{\dR}} \ar{d} & \Rep_{\BdR^+}^\fl(\GG_K)^{\leq 0} \ar{d}\\
			\Rep_{\BdR}^\fl(\GG_K) \ar{r}{\id} & \Rep_{\BdR}^\fl(\GG_K),
		\end{tikzcd}
	\]
	where the vertical arrows are the functors of extension of scalars, is commutative.
\end{coro}

We deduce the following description of the universal morphism associated with \(\trunc_{\dR}\) from Corollary~\ref{coro:universalmorphFil}.
\begin{coro} \label{coro:universalmorphdR}
	Let \(\EE_\dR^+\) be a generically flat \(\BdR^+\)\=/representation of \(\GG_K\).
	The inclusion map
	\[
		i_{\EE_\dR^+} : \EE_\dR^+ \subset \EE_\dR^+ + \BdR^+ \otimes_K (\EE_\dR)^{\GG_K}
	\]
	is the universal morphism from \(\EE_\dR^+\) to the forgetful functor from \(\Rep_{\BdR^+}^\fl(\GG_K)^{\leq 0}\) to \(\Rep_{\BdR^+}^\fl(\GG_K)\).
\end{coro}

\subsection{Modification of de Rham vector bundles} \label{subsec:truncdR}
Let \(\Bun(\GG_K)_\dR^{\leq 0}\) be the full subcategory of \(\Bun(\GG_K)_\dR\) of \(\GG_K\)\=/equivariant de Rham vector bundles over \(\XFF\) whose Hodge-Tate weights are less than or equal to \(0\).
We use the results from the previous subsection to construct a left adjoint to the forgetful functor from \(\Bun(\GG_K)_\dR^{\leq 0}\) to \(\Bun(\GG_K)_\dR\).

We consider the functor
\begin{equation} \label{eq:prototruncFib}
	\begin{split}
		(\id, \trunc_{\dR}) : \Rep_{\Be}(\GG_K) \times \Rep_{\BdR^+}^\fl(\GG_K) & \ra \Rep_{\Be}(\GG_K) \times \Rep_{\BdR^+}^\fl(\GG_K)^{\leq 0}\\
		(\EE_\e, \EE_\dR^+) & \mapsto (\EE_\e, \EE_\dR^+ + \BdR^+ \otimes_K (\EE_\dR)^{\GG_K})
	\end{split}
\end{equation}
formed by the identity functor on \(\Rep_{\Be}(\GG_K)\) together with the truncation functor \(\trunc_{\dR}\).
Under the equivalence of Theorem~\ref{theo:BeauvilleLaszlo}, the category \(\Bun(\GG_K)_\dR\) has been identified with the fibre product of
\[
	\begin{tikzcd}
		& \Rep_{\BdR^+}^\fl(\GG_K) \ar{d}\\
		\Rep_{\Be}(\GG_K) \ar{r} & \Rep_{\BdR}(\GG_K),
	\end{tikzcd}
\]
and the category \(\Bun(\GG_K)_\dR^{\leq 0}\) is identified with the fibre product of
\[
	\begin{tikzcd}
		& \Rep_{\BdR^+}^\fl(\GG_K)^{\leq 0} \ar{d}\\
		\Rep_{\Be}(\GG_K) \ar{r} & \Rep_{\BdR}(\GG_K).
	\end{tikzcd}
\]
Moreover, by Corollary~\ref{coro:truncRep}, the diagram
\[
		\begin{tikzcd}
			\Rep_{\BdR^+}^\fl(\GG_K) \ar{r}{\trunc_{\dR}} \ar{d} & \Rep_{\BdR^+}^\fl(\GG_K)^{\leq 0} \ar{d}\\
			\Rep_{\BdR}(\GG_K) \ar{r}{\id} & \Rep_{\BdR}(\GG_K)
		\end{tikzcd}
\]
is commutative.
Therefore, the functor~\eqref{eq:prototruncFib} induces a well-defined functor
\begin{equation} \label{eq:truncBunDef}
	\Bun(\GG_K)_\dR \ra \Bun(\GG_K)_\dR^{\leq 0},
\end{equation}
which associates with a \(\GG_K\)\=/equivariant de Rham vector bundle \(\EE=(\EE_\e, \EE_\dR^+, \iota_{\EE})\) the vector bundle
\[
		(\EE_\e, \EE_\dR^+ + \BdR^+ \otimes_K \DdR(\EE), \iota_{\EE}).
\]

\begin{defi}
	Let \(\EE=(\EE_\e, \EE_\dR^+, \iota_{\EE})\) be a de Rham \(\GG_K\)\=/equivariant vector bundle over \(\XFF\).
	The vector bundle \((\EE_\e, \EE_\dR^+ + \BdR^+ \otimes_K \DdR(\EE), \iota_{\EE})\) is called the \emph{modification by truncation} of \(\EE\) and is denoted by
	\[
		\EE_+ = (\EE_\e, \EE_\dR^+ + \BdR^+ \otimes_K \DdR(\EE), \iota_{\EE}).
	\]
\end{defi}

\begin{defi}
	The functor~\eqref{eq:truncBunDef} is called the \emph{truncation functor}, and is denoted by
	\[
		\begin{split}
			\trunc_{\HT} : \Bun(\GG_K)_\dR & \ra \Bun(\GG_K)_\dR^{\leq 0}\\
			\EE & \mapsto \EE_+.
		\end{split}
	\]
\end{defi}

The following commutative diagram illustrates the definition of the truncation functor:
\[
	\begin{tikzcd}
		\Bun(\GG_K)_\dR \ar{rr} \ar{dd} \ar{dr}[swap]{\trunc_{\HT}}                         	&                                                                                     	& \Rep_{\Be}(\GG_K)  \ar{dd} \ar{rd}{\id}      &                           \\
																						& \Bun(\GG_K)_\dR^{\leq 0}  \ar[crossing over]{rr}                                    	&                                              & \Rep_{\Be}(\GG_K) \ar{dd} \\
		\Rep_{\BdR^+}^\fl(\GG_K) \ar{rr} \ar{dr}[swap]{\trunc_{\dR}}                       	&                                                                                     	& \Rep_{\BdR}(\GG_K) \ar[from=dd] \ar{rd}{\id} &                           \\
																						& \Rep_{\BdR^+}^\fl(\GG_K)^{\leq 0} \ar[from=uu,crossing over] \ar[crossing over]{rr} 	&                                              & \Rep_{\BdR}(\GG_K)        \\
		\Fil_K \ar{rr} \ar{uu}[sloped]{\sim} \ar{dr}[swap]{\trunc_{\Fil}} 	&                                                                             		    & \Vec_K   \ar{dr}{\id}                        &                           \\
																						& \Fil_K^{\leq 0} \ar[crossing over]{uu}[sloped, near start]{\sim} \ar{rr}		&                                              & \Vec_K.  \ar{uu}
	\end{tikzcd}
\]

From Corollary~\ref{coro:truncRep} follows:
\begin{prop} \label{prop:truncFib}
	The truncation functor \(\trunc_{\HT}\) is exact and left adjoint to the forgetful functor.
\end{prop}

If \(\EE\) is a de Rham \(\GG_K\)\=/equivariant vector bundle, then, by construction, we have
\[
	(\DdR(\EE_+), \Fil \DdR(\EE_+)) \simeq (\DdR(\EE), \Fil_+ \DdR(\EE)).
\]
We note that if the Hodge-Tate weights of \(\EE\) are less than or equal to \(0\), then \((\DdR(\EE), \Fil_+ \DdR(\EE)) = (\DdR(\EE), \Fil \DdR(\EE))\), and \(\EE_+ \simeq \EE\).
Precisely:
\begin{coro} \label{coro:truncFib}
	The composition of the truncation functor with the forgetful functor
	\[
		\Bun^{\leq 0}(\GG_K)_\dR \ra \Bun(\GG_K)_\dR \xra{\trunc_{\HT}} \Bun^{\leq 0}(\GG_K)_\dR
	\]
	is isomorphic to the identity functor on \(\Bun^{\leq 0}(\GG_K)_\dR\).
\end{coro}
\begin{proof}
	The assertion is a formal consequence of the adjunction property from Proposition~\ref{prop:truncFib} and the fully faithfulness of the forgetful functor (see \cite[IV \S 3 Theorem~1]{MacLane98}).
\end{proof}

	From Corollary~\ref{coro:universalmorphdR}, we obtain a description of the universal morphism associated with \(\trunc_{\HT}\).
	If \(\EE=(\EE_\e, \EE_\dR^+, \iota_{\EE})\) is a \(\GG_K\)\=/equivariant de Rham vector bundle over \(\XFF\),
	then the identity map \(\id_{\EE_\e}\) on \(\EE_\e\) together with the natural inclusion
	\[
		i_{\EE_\dR^+} : \EE_\dR^+ \subset \EE_\dR^+ + \BdR^+ \otimes_K \DdR(\EE)
	\]
	define an injective morphism of \(\GG_K\)\=/equivariant vector bundles
	\[
		\eta_\EE = (\id_{\EE_\e}, i_{\EE_\dR^+}) : \EE \ra \EE_+.
	\]
\begin{coro} \label{coro:universalmorphFib}
	Let \(\EE=(\EE_\e, \EE_\dR^+, \iota_{\EE})\) be a \(\GG_K\)\=/equivariant de Rham vector bundle over \(\XFF\).
	The map \(\eta_\EE\) is the universal morphism from \(\EE\) to the forgetful functor from \(\Bun^{\leq 0}(\GG_K)_\dR\) to \(\Bun(\GG_K)_\dR\).
\end{coro}

\subsection{Hodge-Tate and Harder-Narasimhan filtrations} \label{subsec:HTHN}
If \(V\) is a de Rham \(p\)\=/adic representation of \(\GG_K\), we recall (Theorem~\ref{theo:RepBun0} and Remark~\ref{rema:DdRRep}) that the associated vector bundle
\[
	\EE(V) = (\Be \otimes_{\Qp} V, \BdR^+ \otimes_{\Qp} V)
\]
is \(\GG_K\)\=/equivariant, de Rham and semi-stable of slope \(0\).
In particular, we can consider the modification by truncation of \(\EE(V)\) which we denote by
\[
	\EE_+(V) = \left(\Be \otimes_{\Qp} V, \BdR^+ \otimes_{\Qp} V + \BdR^+ \otimes_K \DdR(V)\right).
\]
We also denote by
\[
	\eta_V : \EE(V) \ra \EE_+(V).
\]
the universal morphism from Corollary~\ref{coro:universalmorphFib}.
The cokernel of \(\eta_V\), which we denote by \(\FF_+(V)\), is a \(\GG_K\)\=/equivariant torsion coherent sheaf defined by
\begin{equation} \label{eq:Fplus}
	\FF_+(V) = \left( 0, \frac{\BdR^+ \otimes_{\Qp} V + \BdR^+ \otimes_K \DdR(V)}{\BdR^+ \otimes_{\Qp} V}\right),
\end{equation}
and there is a short exact sequence of \(\GG_K\)\=/equivariant coherent sheaves
\begin{equation} \label{eq:EV}
	0 \ra \EE(V) \xra{\eta_V} \EE_+(V) \ra \FF_+(V) \ra 0.
\end{equation}

In the remainder of this section, we study the Harder\-/Narasimhan slopes of \(\EE_+(V)\).

\begin{lemm} \label{lemm:HN0}
	Let \(V\) be a de Rham \(p\)\=/adic representation of \(\GG_K\).
	The Harder\-/Narasimhan slopes of \(\EE_+(V)\) are greater than or equal to \(0\).
\end{lemm}
\begin{proof}
	The short exact sequence~\eqref{eq:EV} induces the cohomological exact sequence
	\begin{equation} \label{eq:cohomHN0}
		\H^1(\XFF,\EE(V)) \ra \H^1(\XFF,\EE_+(V)) \ra \H^1(\XFF,\FF_+(V)) \ra 0.
	\end{equation}
	Since \(\EE(V)\) and \(\FF_+(V)\) are semi-stable of slopes \(0\) and \(+\infty\) respectively, the groups \(\H^1(\XFF,\EE(V))\) and \(\H^1(\XFF,\FF_+(V))\) vanish by Theorem~\ref{theo:cohomHN}.
	Therefore, the exact sequence~\eqref{eq:cohomHN0} forces the group \(\H^1(\XFF,\EE_+(V))\) to be trivial, and the lemma follows from another application of Theorem~\ref{theo:cohomHN}.
\end{proof}

\begin{defi} \label{defi:V0}
	Let \(V\) be a de Rham \(p\)\=/adic representation of \(\GG_K\).
	The representation \(V_0\) is the minimal sub\=/\(\GG_K\)\=/representation of \(V\) such that the Hodge-Tate weights of \(V/V_0\) are less than or equal to \(0\).
\end{defi}

\begin{rema}
	The representation \(V_0\) is well-defined.
	Indeed, if \(U\) and \(W\) are sub\=/\(\GG_K\)\=/representations of \(V\) such that the Hodge-Tate weights of both \(V/U\) and \(V/W\) are less than or equal to \(0\), then \(V/(U \cap W)\) injects into \(V/U \oplus V/W\) which implies, by Proposition~\ref{prop:DdR} and Remark~\ref{rema:DdR}, that the Hodge-Tate weights of \(V/(U \cap W)\) are less than or equal to \(0\).
\end{rema}

\begin{lemm} \label{lemm:V0}
	Let \(V\) be a de Rham \(p\)\=/adic representation of \(\GG_K\).
	There exists no proper sub\=/\(\GG_K\)\=/representation \(V_0^\prime\) of \(V_0\) such that the Hodge-Tate weights of \(V_0/V_0^\prime\) are less than or equal to \(0\), \ie
	\[
		(V_0)_0=V_0.
	\]
\end{lemm}
\begin{proof}
	If \(V_0^\prime\) is a sub\=/\(\GG_K\)\=/representation of \(V_0\) such that the Hodge-Tate weights of \(V_0/V_0^\prime\) are less than or equal to \(0\), then we have the short exact sequence
	\begin{equation} \label{eq:seqV0}
		0 \ra V_0/V_0^\prime \ra V/V_0^\prime \ra V/V_0 \ra 0.
	\end{equation}

	On the one hand, by hypothesis, the Hodge-Tate weights of \(V_0/V_0^\prime\) are less than or equal to \(0\).
	On the other hand, the Hodge-Tate weights of \(V/V_0\) are less than or equal to \(0\) by definition of \(V_0\).
	Thus, by Proposition~\ref{prop:DdR} and Remark~\ref{rema:DdR}, the short exact sequence~\eqref{eq:seqV0} forces the Hodge-Tate weights of \(V/V_0^\prime\) to be less than or equal to \(0\).
	By minimality of \(V_0\), we then have \(V_0^\prime = V_0\).
\end{proof}

\begin{prop} \label{prop:HN0V0}
	Let \(V\) be a de Rham \(p\)\=/adic representation of \(\GG_K\).
	The Harder\-/Narasimhan slopes of \(\EE_+(V)\) are strictly greater than \(0\) if and only if \(V=V_0\).
\end{prop}
\begin{proof}[Proof. The condition \(V=V_0\) is sufficient.]
	We first prove that the condition \(V=V_0\) is sufficient.
	We proceed by contradiction, so we assume that \(V=V_0\) and that \(0\) is a Harder\-/Narasimhan slope of \(\EE_+(V)\).
	Therefore, by Lemma~\ref{lemm:HN0}, the smallest Harder\-/Narasimhan slope of \(\EE_+(V)\) is \(0\).
	Let
	\begin{equation} \label{eq:HN1step}
		0 \ra \EE^\prime \ra \EE_+(V) \ra \EE^\dprime \ra 0
	\end{equation}
	be the first step of the Harder\-/Narasimhan filtration of \(\EE_+(V)\), so the vector bundle \(\EE^\dprime\) is \(\GG_K\)\=/equivariant and semi-stable of slope \(0\).
	Thus, by Theorem~\ref{theo:RepBun0}, there exists a \(p\)\=/adic representation \(V^\dprime\) of \(\GG_K\) such that
	\[
		\EE^\dprime = \EE(V^\dprime).
	\]

	Since \(\EE_+(V)\) is de Rham with Hodge-Tate weights less than or equal to \(0\), Proposition~\ref{prop:DdR} and Remark~\ref{rema:DdR} applied to the exact sequence~\eqref{eq:HN1step} imply that the representation \(V^\dprime\) is de Rham with Hodge-Tate weights less than or equal to \(0\).

	Therefore, the adjunction property of the truncation functor (Proposition~\ref{prop:truncFib}) gives
	\begin{equation} \label{eq:adjonction}
		\Hom(\EE_+(V), \EE(V^\dprime)) \simeq \Hom(\EE(V), \EE(V^\dprime)).
	\end{equation}
	By Theorem~\ref{theo:RepBun0}, the equation~\eqref{eq:adjonction} gives
	\begin{equation} \label{eq:rep}
		\Hom(\EE_+(V), \EE(V^\dprime)) \simeq \Hom_{\Rep_{\Qp}(\GG_K)}(V, V^\dprime).
	\end{equation}
	Since \(V=V_0\) and the Hodge-Tate weights of \(V^\dprime\) less than or equal to \(0\), by Lemma~\ref{lemm:V0}, we have
	\begin{equation} \label{eq:HomNULL}
		\Hom_{\Rep_{\Qp}(\GG_K)}(V, V^\dprime)=0.
	\end{equation}
	We reach a contradiction: the combination of the equations~\eqref{eq:rep} and \eqref{eq:HomNULL} yields
	\[
		\Hom(\EE_+(V), \EE(V^\dprime)) = 0,
	\]
	which contradicts the existence of the exact sequence~\eqref{eq:HN1step}, and consequently, the assumption that \(0\) is a Harder\-/Narasimhan slope of \(\EE_+(V)\).
\end{proof}

\begin{proof}[Proof. The condition \(V=V_0\) is necessary.]
	We prove that the condition \(V=V_0\) is necessary.
	We proceed by contraposition, so we assume \(V \neq V_0\).
	By exactness of the truncation functor (Proposition~\ref{prop:truncFib}), the short exact sequence
	\[
		0 \ra V_0 \ra V \ra V/V_0 \ra 0
	\]
	induces a short exact sequence
	\begin{equation} \label{eq:HNV0}
		0 \ra \EE_+(V_0) \ra \EE_+(V) \ra \EE_+(V/V_0) \ra 0.
	\end{equation}

	On the one hand, since the Hodge-Tate weights of \(V/V_0\) are less than or equal to \(0\) by definition of \(V_0\), by Corollary~\ref{coro:truncFib}, we have \(\EE_+(V/V_0) \simeq \EE(V/V_0)\); in particular, the vector bundle \(\EE_+(V/V_0)\) is semi-stable of slope \(0\).
	On the other hand, since \((V_0)_0 = V_0\) (Lemma~\ref{lemm:V0}), the already proved part of Proposition~\ref{prop:HN0V0} implies that the Harder\-/Narasimhan slopes of \(\EE_+(V_0)\) are strictly greater than \(0\).
	The unicity of the Harder\-/Narasimhan filtration implies that the short exact sequence~\eqref{eq:HNV0} is the first step of the Harder\-/Narasimhan filtration of \(\EE_+(V)\).
	In particular, \(0\) is one of the Harder\-/Narasimhan slopes of \(\EE_+(V)\).
\end{proof}

\begin{coro} \label{coro:HN}
	Let \(V\) be a de Rham \(p\)\=/adic representation of \(\GG_K\).
	The short exact sequence
	\[
		0 \ra \EE_+(V_0) \ra \EE_+(V) \ra \EE_+(V/V_0) \ra 0
	\]
	is the first step of the Harder-Narasimhan filtration of \(\EE_+(V)\) with
	\[
		\EE_+(V/V_0) \simeq \EE(V/V_0)
	\]
	semi-stable of slope \(0\).
\end{coro}

\section{Universal norms of \texorpdfstring{\(p\)\=/}{p-}adic Galois representations} \label{sec:universalnorms}
In this final section, we prove the main results of the article regarding universal norms of \(p\)\=/adic Galois representations.
To do so, we first recall the definition of the group \(E_\discu(V/T)\) and its Hausdorff completion \(E_+(V/T)\), associated by Fontaine~\cite[\S 8]{Fontaine03} with a \(p\)\=/adic representation \(V\) of \(\GG_K\) and a \(\Zp\)\=/lattice \(T\) in \(V\) stable under the action of \(\GG_K\).
These groups are respectively related to the Bloch-Kato subgroups (see \S \ref{subsec:BK}) and the modification by truncation from the previous section (see \S \ref{subsec:almostCpRep}).
A study of their Galois cohomology will implies the main results.

\subsection{Almost \texorpdfstring{\(\Cp\)\=/}{Cp-}representations and groups of points} \label{subsec:points}
Let \(V\) be a \(p\)\=/adic representation of \(\GG_K\).
The \emph{tangent space} of \(V\) is the finite dimensional \(K\)\=/vector space
\[
	t_V = \left((\BdR/\BdR^+) \otimes_{\Qp} V \right)^{\GG_K}.
\]
For all \(K\)\=/algebra \(R\), we set \(t_V(R) = R \otimes_K t_V\).
The inclusion
\[
	t_V \subset (\BdR/\BdR^+) \otimes_{\Qp} V
\]
induces a \(\Qpbar\)\=/linear map
\begin{equation} \label{eq:tVQpbar}
	t_V(\Qpbar) \ra (\BdR/\BdR^+) \otimes_{\Qp} V,
\end{equation}
and a \(\BdR^+\)\=/linear map
\begin{equation} \label{eq:tVBdR}
	t_V(\BdR^+) \ra (\BdR/\BdR^+) \otimes_{\Qp} V.
\end{equation}
The first map~\eqref{eq:tVQpbar} is injective, and we use it to identify \(t_V(\Qpbar)\) with a sub\=/\(\Qpbar\)\=/vector space of \((\BdR/\BdR^+)\otimes_{\Qp} V\) stable under the action of \(\GG_K\).
The image of the second map~\eqref{eq:tVBdR}, which is denoted by \(\h{t}_V(\Qpbar)\), identifies with the topological closure of \(t_V(\Qpbar)\) in \((\BdR/\BdR^+)\otimes_{\Qp} V\).
Moreover, \(\h{t}_V(\Qpbar)\) is a \(\BdR^+\)\=/representation of \(\GG_K\) whose underlying \(\BdR^+\)\=/module is of finite length.

The fundamental exact sequence~\eqref{eq:funda} tensored with \(V\) gives a short exact sequence
\[
	0 \ra V \ra E_\e(V) \ra (\BdR/\BdR^+) \otimes_{\Qp} V \ra 0,
\]
where we set \(E_\e(V)=\Be \otimes_{\Qp} V\).
The group \(E_\disc(V)\) (respectively \(E_+(V)\)) is defined as the inverse image of \(t_V(\Qpbar)\) (respectively \(\h{t}_V(\Qpbar)\)) in \(E_\e(V)\), so that there is a commutative diagram with exact rows
\begin{equation} \label{eq:EastV}
	\begin{tikzcd}
		0 \ar{r} & V \ar{r} & E_\e(V) \ar{r} & (\BdR/\BdR^+) \otimes_{\Qp} V \ar{r} & 0\\
		0 \ar{r} & V \ar{r} \ar[equal]{u} & E_+(V) \ar{r} \ar[hook]{u} & \h{t}_V(\Qpbar) \ar{r} \ar[hook]{u} & 0\\
		0 \ar{r} & V \ar{r} \ar[equal]{u} & E_\disc(V) \ar{r} \ar[hook]{u} & t_V(\Qpbar) \ar{r} \ar[hook]{u} & 0.
	\end{tikzcd}
\end{equation}

Let \(T\) be a \(\Zp\)\=/lattice in \(V\) stable under the action of \(\GG_K\).
For each \(\ast \in \{\e, +, \disc\}\), let
\[
	E_\ast(V/T) = E_\ast(V)/T.
\]
The diagram~\eqref{eq:EastV} induces a commutative diagram with exact rows
\[
	\begin{tikzcd}
		0 \ar{r} & V/T \ar{r} & E_\e(V/T) \ar{r} & (\BdR/\BdR^+) \otimes_{\Qp} V \ar{r} & 0\\
		0 \ar{r} & V/T \ar{r} \ar[equal]{u} & E_+(V/T) \ar{r} \ar[hook]{u} & \h{t}_V(\Qpbar) \ar{r} \ar[hook]{u} & 0\\
		0 \ar{r} & V/T \ar{r} \ar[equal]{u} & E_\disc(V/T) \ar{r} \ar[hook]{u} & t_V(\Qpbar) \ar{r} \ar[hook]{u} & 0.
	\end{tikzcd}
\]

The groups \(E_\disc(V/T)\) and \(E_+(V/T)\) are characterized as sub\-/topological\-/\(\GG_K\)\=/modules of \(E_\e(V/T)\) \cite[Proposition~8.4]{Fontaine03}.
\begin{prop}[Fontaine] \label{prop:Ediscplus}
	The group \(E_\disc(V/T)\) is the maximal subgroup stable under the action of \(\GG_K\) of \(E_\e(V/T)\) on which the action of \(\GG_K\) is discrete.
	The group \(E_+(V/T)\) is the topological closure of \(E_\disc(V/T)\) in \(E_\e(V/T)\).
\end{prop}
\begin{rema}
	The topological \(\GG_K\)\=/module \(E_+(V)\) is an \emph{almost \(\Cp\)\=/representation of \(\GG_K\)}, a category introduced by Fontaine~\cite{Fontaine03}.
	In particular, \(E_+(V)\) is a \(p\)\=/adic Banach space endowed with a continuous and linear action of \(\GG_K\).
	As such, \(E_+(V)\) contains a sub\=/\(\Zp\)\=/module \(\SS_+\) which is separated and complete for the \(p\)\=/adic topology, stable under the action of \(\GG_K\), and such that \(E_+(V) = \SS_+ \otimes_{\Zp} \Qp\) with the corresponding topology, \ie \(\{p^n \SS_+, n \in \N \}\) is a fundamental system of neighbourhood of \(0\) in \(E_+(V)\).
	Let \(S\) be the image of \(\SS_+ \cap E_\disc(V)\) in \(E_\disc(V/T)\).
	Since \(E_\disc(V/T)\) is dense in \(E_+(V/T)\), there is an isomorphism of topological \(\GG_K\)\=/modules
	\[
		\limproj_{n \in \N} E_\disc(V/T)/p^n S \simeq E_+(V/T),
	\]
	where the left-hand side is equipped with the inverse limit topology and each term \(E_\disc(V/T)/p^n S\) is discrete.
	Hence, \(E_+(V/T)\) is the Hausdorff completion of \(E_\disc(V/T)\).
\end{rema}

The topological \(\GG_K\)\=/module \(E_\disc(V/T)\) is the \emph{group of points with values in \(\Qpbar\)} associated with \(V/T\).
By Proposition~\ref{prop:Ediscplus}, the action of \(\GG_K\) on \(E_\disc(V/T)\) is compatible with the discrete topology.
\begin{defi}
	The topological \(\GG_K\)\=/module \(E_\discu(V/T)\) is defined as the \(\GG_K\)\=/module \(E_\disc(V/T)\) endowed with the discrete topology.
\end{defi}
We will abusively also denote by \(t_V(\Qpbar)\) the \(\GG_K\)\=/module \(t_V(\Qpbar)\) equipped with the discrete topology, so that \(E_\discu(V/T)\) fits into the short exact sequence of discrete \(\GG_K\)\=/modules
\[
	0 \ra V/T \ra E_\discu(V/T) \ra t_V(\Qpbar) \ra 0.
\]

\begin{rema} \label{rema:VarAbel}
	Let \(A\) be an Abelian variety defined over \(K\).
	Let \(t_A\) be the tangent space of \(A\).
	Let \(A[p^n]\) be the subgroup of \(p^n\)\=/torsion points of \(A(\Cp)\).
	The \(p\)\=/adic Tate module \(T_p(A) = \limproj_{\times p} A[p^n]\) is a free \(\Zp\)\=/module of rank \(2\dim(A)\) equipped with a continuous and linear action of \(\GG_K\).
	Thus, \(V_p(A) = \Qp \otimes_{\Zp} T_p(A)\) is a \(p\)\=/adic representation of \(\GG_K\), and \(V_p(A)/T_p(A) = \bigcup_{n \in \N} A[p^n] = A[p^\infty]\) is the discrete \(\GG_K\)\=/module of \(p\)\=/power torsion points of \(A(\Cp)\).
	The exponential map on \(A\) enables to construct \(\GG_K\)\=/equivariant splittings
	\[
		\begin{split}
			A(\Cp) & = A^{(p)}(\Cp) \oplus A[p^\prime],\\
			A(\Qpbar) & = A^{(p)}(\Qpbar) \oplus A[p^\prime],
		\end{split}
	\]
	where \(A[p^\prime]\) denotes the subgroup of prime-to\=/\(p\) torsion points of \(A(\Cp)\) (see \cite[\S 1.1]{Fontaine03}).
	Fontaine proved~\cite[Proposition~8.6]{Fontaine03} that there are canonical and functorial isomorphisms of topological \(\GG_K\)\=/modules
	\begin{align*}
		t_{V_p(A)}(\Qpbar) & \simeq t_A(\Qpbar),& \h{t}_{V_p(A)}(\Qpbar) & \simeq t_A(\Cp),\\
		E_\disc(A[p^\infty]) & \simeq A^{(p)}(\Qpbar),& E_+(A[p^\infty]) & \simeq A^{(p)}(\Cp).
	\end{align*}

	Fontaine also proved analogous results for \(p\)\=/divisible groups.
	Let \(G\) be a \(p\)\=/divisible group defined over \(\OO_K\).
	Let \(t_G\) be the tangent space of \(G\).
	Let \(T_p(G)\) be the \(p\)\=/adic Tate module of \(G\), and let \(V_p(G)=\Qp \otimes_{\Zp} T_p(G)\), so that \(V_p(G)/T_p(G)=G[p^\infty]\) is the discrete \(\GG_K\)\=/module of \(p\)\=/power torsion points of \(G(\OO_{\Cp})\).
	There are canonical and functorial isomorphisms of topological \(\GG_K\)\=/modules
	\begin{align*}
		t_{V_p(G)}(\Qpbar) & \simeq t_G(\Qpbar), & \h{t}_{V_p(G)}(\Qpbar) & \simeq t_G(\Cp),\\
		E_\disc(G[p^\infty]) & \simeq G(\OO_{\Qpbar}),& E_+(G[p^\infty]) & \simeq G(\OO_{\Cp}).
	\end{align*}
\end{rema}

\subsection{The Bloch-Kato subgroups} \label{subsec:BK}
We recall the definition of the Bloch-Kato subgroups~\cite[\S 3]{BlochKato90}.
Let \(V\) be a \(p\)\=/adic representation of \(\GG_K\).
For each finite extension \(K^\prime\) of \(K\), the \emph{exponential}, \emph{finite} and \emph{geometric Bloch-Kato subgroups} of \(\H^1(K^\prime,V)\) are respectively defined by
\[
	\begin{split}
		\H^1_e(K^\prime,V) & = \Ker\left(\H^1(K^\prime,V) \ra \H^1(K^\prime,\Be \otimes_{\Qp} V)\right),\\
		\H^1_f(K^\prime,V) & = \Ker\left(\H^1(K^\prime,V) \ra \H^1(K^\prime,\Bcris \otimes_{\Qp} V)\right),\\
		\H^1_g(K^\prime,V) & = \Ker\left(\H^1(K^\prime,V) \ra \H^1(K^\prime,\BdR \otimes_{\Qp} V)\right).
	\end{split}
\]
There are inclusions
\[
	\H^1_e(K^\prime, V) \subset \H^1_f(K^\prime, V) \subset \H^1_g(K^\prime, V) \subset \H^1(K^\prime, V).
\]

Let \(T\) be a \(\Zp\)\=/lattice in \(V\) stable under the action of \(\GG_K\).
The short exact sequence
\[
	0 \ra T \ra V \ra V/T \ra 0
\]
induces a cohomological exact sequence
\[
	\H^1(K^\prime,T) \xra{\alpha_{K^\prime}} \H^1(K^\prime,V) \xra{\beta_{K^\prime}} \H^1(K^\prime,V/T).
\]
For \(\ast \in \{e,f,g\}\), the Bloch-Kato subgroups of \(\H^1(K^\prime,T)\) and \(\H^1(K^\prime,V/T)\) are respectively defined by
\[
	\begin{split}
		\H^1_\ast(K^\prime,T) & = \alpha_{K^\prime}^{-1}(\H^1_\ast(K^\prime,V)),\\
		\H^1_\ast(K^\prime,V/T) & = \beta_{K^\prime}(\H^1_\ast(K^\prime,V)).
	\end{split}
\]

For each \(\ast \in \{e,f,g\}\), the groups \(\H^1_\ast(K^\prime,V/T)\) are compatible under the restriction maps.
If \(L\) is an algebraic extension of \(K\), then the Bloch-Kato subgroups of \(\H^1(L,V/T)\) are defined by
\[
	\H^1_\ast(L,V/T) = \liminj_{\res, K^\prime} \H^1_\ast(K^\prime, V/T)
\]
where the limit is taken relatively to the restriction maps and \(K^\prime\) runs through the finite extensions of \(K\) contained in \(L\).

The following proposition, which relates the exponential Bloch-Kato subgroup to \(E_\discu(V/T)\), is essentially contained in~\cite[\S 8]{Fontaine03} or \cite[Corollary~3.8.4]{BlochKato90}.
\begin{prop} \label{prop:BK}
	Let \(L\) be an algebraic extension of \(K\).
	The short exact sequence
	\[
		0 \ra V/T \ra E_\discu(V/T) \ra t_V(\Qpbar) \ra 0
	\]
	induces a cohomological short exact sequence
	\[
		0 \ra \H^1_e(L,V/T) \ra \H^1(L,V/T) \ra \H^1(L,E_\discu(V/T)) \ra 0.
	\]
\end{prop}
\begin{proof}
	Let \(K^\prime\) be a finite extension of \(K\).
	Each short exact sequence of the diagram~\eqref{eq:EastV} induces the same cohomological exact sequence
	\[
		\begin{tikzcd}
			0 \ar{r} & V^{\GG_{K^\prime}} \ar{r} & E_\e(V)^{\GG_{K^\prime}} \ar{r} & \left((\BdR/\BdR^+) \otimes_{\Qp} V\right)^{\GG_{K^\prime}} \ar{r} & \H^1(K^\prime,V)\\
			0 \ar{r} & V^{\GG_{K^\prime}} \ar{r} \ar[equal]{u} & E_\disc(V)^{\GG_{K^\prime}} \ar{r} \ar[equal]{u} & t_V(K^\prime) \ar{r} \ar[equal]{u} & \H^1(K^\prime,V). \ar[equal]{u}
		\end{tikzcd}
	\]
	Hence, by definition of the exponential Bloch-Kato subgroups of \(\H^1(K^\prime,V)\), there is an exact sequence
	\begin{equation} \label{eq:EdiscHe}
		0 \ra  V^{\GG_{K^\prime}} \ra E_\disc(V)^{\GG_{K^\prime}} \ra t_V(K^\prime) \ra \H^1_e(K^\prime,V) \ra 0.
	\end{equation}
	The commutative diagram with exact rows
	\[
		\begin{tikzcd}
			0 \ar{r} & V \ar{r} \ar{d} & E_\disc(V) \ar{r} \ar{d} & t_V(\Qpbar) \ar{r} \ar[equal]{d} & 0 \\
			0 \ar{r} & V/T \ar{r} & E_\disc(V/T) \ar{r} & t_V(\Qpbar) \ar{r} & 0,
		\end{tikzcd}
	\]
	induces the exact commutative diagram
	\begin{equation} \label{eq:He}
		\begin{tikzcd}
			0 \ar{r} & V^{\GG_{K^\prime}} \ar{r} \ar{d} & E_\disc(V)^{\GG_{K^\prime}} \ar{r} \ar{d} & t_V(K^\prime) \ar{r} \ar[equal]{d} & \H^1(K^\prime,V) \ar{d}{\beta_{K^\prime}} \\
			0 \ar{r} & (V/T)^{\GG_{K^\prime}} \ar{r} & E_\disc(V/T)^{\GG_{K^\prime}} \ar{r} & t_V(K^\prime) \ar{r} & \H^1(K^\prime,V/T).
		\end{tikzcd}
	\end{equation}
	The definition of \(\H^1_e(K^\prime,V/T)= \beta_{K^\prime}(\H^1_e(K^\prime,V))\) together with the exact sequence~\eqref{eq:EdiscHe} and the commutativity of the diagram~\eqref{eq:He} gives an exact sequence
	\begin{equation} \label{eq:Heleft}
		0 \ra  (V/T)^{\GG_{K^\prime}} \ra E_\disc(V/T)^{\GG_{K^\prime}} \ra t_V(K^\prime) \ra \H^1_e(K^\prime,V/T) \ra 0.
	\end{equation}
	This last exact sequence~\eqref{eq:Heleft} implies that the long cohomological exact sequence associated with the short exact sequence of discrete \(\GG_K\)\=/modules
	\[
		0 \ra V/T \ra E_\discu(V/T) \ra t_V(\Qpbar) \ra 0
	\]
	induces the short exact sequence
	\begin{equation} \label{eq:HeIV}
		0 \ra \H^1_e(K^\prime,V/T) \ra \H^1(K^\prime,V/T) \ra \H^1(K^\prime,E_\discu(V/T)) \ra 0,
	\end{equation}
	where the final zero follows from the vanishing of \(\H^1(K^\prime, \Qpbar)\).
	Since \(V/T\) and \(E_\discu(V/T)\) are discrete \(\GG_K\)\=/modules, we deduce the lemma for a general algebraic extension \(L\) of \(K\) by taking the limit relative to the restriction maps of the exact sequence \eqref{eq:HeIV}.
\end{proof}

\begin{rema} \label{rema:Kummer}
	Let \(G\) be a \(p\)\=/divisible group defined over \(\OO_K\).
	Let \(L\) be an algebraic extension of \(K\).
	The logarithm short exact sequence~\cite{Tate67}
	\[
		0 \ra G[p^\infty] \ra G(\OO_{\Qpbar}) \xra{\log_G} t_G(\Qpbar) \ra 0
	\]
	induces the cohomological short exact sequence
	\[
		0 \ra G(\OO_L) \otimes_{\Zp} \Qp/\Zp \xra{\kappa_L^G} \H^1(L,G[p^\infty]) \ra \H^1(L,G) \ra 0.
	\]
	The application \(\kappa_L^G\) is the Kummer map.
	By combining Proposition~\ref{prop:BK} with Fontaine's results from Remark~\ref{rema:VarAbel}, we recover the fact~\cite[Example~3.10]{BlochKato90} that the image of the Kummer map \(\kappa_L^G\) coincides with the exponential Bloch-Kato subgroups \(\H^1_e(L,G[p^\infty])\).

	Similarly, for an Abelian variety \(A\) defined over \(K\), we recover the fact that the image of the Kummer map
	\[
		\kappa_L^A : A(L) \otimes_{\Zp} \Qp/\Zp \ra \H^1(L,A[p^\infty])
	\]
	coincides with \(\H^1_e(L,A[p^\infty])\) (we recall that for an Abelian variety the Bloch-Kato subgroups are all equal, see~\cite[Example~3.11]{BlochKato90}).

	For a general \(p\)\=/adic representation \(V\) of \(\GG_K\), Proposition~\ref{prop:BK}, precisely the exact sequence~\eqref{eq:EdiscHe}, induces an isomorphism
	\[
		\H^0(L,E_\disc(V/T)) \otimes_{\Zp} \Qp/\Zp\simeq \H^1_e(L,V/T).
	\]
\end{rema}

\subsection{Almost \texorpdfstring{\(\Cp\)\=/}{Cp-}representations and vector bundles} \label{subsec:almostCpRep}
We now establish the link with the modification by truncation from the previous section.
\begin{prop} \label{prop:almostcoh}
	Let \(V\) be a de Rham \(p\)\=/adic representation of \(\GG_K\).
	There is a canonical and functorial isomorphism of topological \(\GG_K\)\=/modules
	\[
		E_+(V) \simeq \Gamma(\XFF, \EE_+(V)).
	\]
\end{prop}
\begin{proof}
	We recall the short exact sequence~\eqref{eq:EV} of \(\GG_K\)\=/equivariant coherent sheaves:
	\[
		0 \ra \EE(V) \xra{\eta_V} \EE_+(V) \ra \FF_+(V) \ra 0.
	\]
	The Harder-Narasimhan slopes of these coherent sheaves are all greater than or equal to \(0\) (Lemma~\ref{lemm:HN0}), and we know (Theorem~\ref{theo:RepBun0}) that
	\[
		\Gamma(\XFF,\EE(V)) \simeq V.
	\]
	Therefore, by Theorems~\ref{theo:cohomCoh} and~\ref{theo:cohomHN}, the cohomology of the exact sequence~\eqref{eq:EV} gives the exact commutative diagram
	\[
		\begin{tikzcd}
			0 \ar{r} & V \ar{r} & E_\e(V) \ar{r} & (\BdR/\BdR^+) \otimes_{\Qp} V \ar{r} & 0\\
			0 \ar{r} & \Gamma(\XFF, \EE(V)) \ar{r} \ar{u}[sloped]{\sim} & \Gamma(\XFF,\EE_+(V)) \ar{r} \ar{u} & \Gamma(\XFF,\FF_+(V)) \ar{r} \ar{u} & 0\\
					& & 0 \ar{u} & 0 \ar{u} & .
		\end{tikzcd}
	\]
	Hence, it is enough to prove that \(\Gamma(\XFF,\FF_+(V))\) is isomorphic to \(\h{t}_V(\Qpbar)\), which is done in the next lemma~\ref{lemm:almostcoh}.
\end{proof}

\begin{lemm} \label{lemm:almostcoh}
	Let \(V\) be a de Rham \(p\)\=/adic representation of \(\GG_K\).
	There is a canonical and functorial isomorphism of topological \(\GG_K\)\=/modules
	\[
		\h{t}_V(\Qpbar) \simeq \Gamma(\XFF,\FF_+(V)).
	\]
\end{lemm}
\begin{proof}
	Since \(V\) is de Rham, we have
	\[
		t_V \simeq \DdR(V)/\Fil^0 \DdR(V).
	\]
	Therefore, we have
	\[
		\begin{split}
			\h{t}_V(\Qpbar) & = \Img\left(t_V(\BdR^+) \ra (\BdR/\BdR^+) \otimes_{\Qp} V\right)\\
							& \simeq \Img\left(\BdR^+ \otimes_K (\DdR(V)/\Fil^0 \DdR(V)) \ra (\BdR/\BdR^+) \otimes_{\Qp} V\right)\\
							& \simeq \Img\left(\frac{\BdR^+ \otimes_K \DdR(V)}{\BdR^+ \otimes_K \Fil^0 \DdR(V)} \ra \frac{\BdR \otimes_{\Qp} V}{\BdR^+ \otimes_{\Qp} V}\right)\\
							& \simeq \frac{\BdR^+ \otimes_{\Qp} V + \BdR^+ \otimes_K \DdR(V)}{\BdR^+ \otimes_{\Qp} V}\\
							& \simeq \Gamma(\XFF, \FF_+(V)),
		\end{split}
	\]
	where the last isomorphism is given by Theorem~\ref{theo:cohomCoh} and the definition~\eqref{eq:Fplus} of \(\FF_+(V)\).
\end{proof}

\begin{rema}
	Fontaine~\cite{Fontaine20} has related the categories of almost \(\Cp\)\=/representations of \(\GG_K\) and \(\GG_K\)\=/equivariant coherent sheaves over \(\XFF\); these two categories are not equivalent, but can be reconstruct out of one another (see also~\cite{LeBras18} for an analogous result without Galois action).
\end{rema}

\begin{coro} \label{coro:Bn}
	Let \(V\) be a de Rham \(p\)\=/adic representation of \(\GG_K\).
	There is an isomorphism of \(\BdR^+\)\=/representations of \(\GG_K\)
	\[
		\h{t}_V(\Qpbar) \simeq \bigoplus_{n \in \N} \B_n^{m_n(V)},
	\]
	where \(m_n(V)\) is the multiplicity of \(n\) as a Hodge-Tate weight of \(V\).
\end{coro}
\begin{proof}
	By Lemma~\ref{lemm:almostcoh}, we have
	\begin{equation} \label{eq:tvDdR}
		\begin{split}
			\h{t}_V(\Qpbar) & \simeq \frac{\BdR^+ \otimes_{\Qp} V + \BdR^+ \otimes_K \DdR(V)}{\BdR^+ \otimes_{\Qp} V}\\
			& \simeq \frac{\BdR^+ \otimes_K \DdR(V)}{\BdR^+ \otimes_{\Qp} V \cap \BdR^+ \otimes_K \DdR(V)}.
		\end{split}
	\end{equation}
	We compute \(\BdR^+ \otimes_{\Qp} V \cap \BdR^+ \otimes_K \DdR(V)\).
	Since \(V\) is de Rham, we have
	\[
		\BdR^+ \otimes_{\Qp} V = \sum_{n \in \Z} \Fil^n \BdR \otimes_K \Fil^{-n} \DdR(V).
	\]
	Therefore, we have
	\begin{equation} \label{eq:cap}
		\begin{split}
			& \BdR^+ \otimes_{\Qp} V \cap \BdR^+ \otimes_K \DdR(V)\\
			\simeq & \left(\sum_{n \in \Z} \Fil^n \BdR \otimes_K \Fil^{-n} \DdR(V)\right) \cap \BdR^+ \otimes_K \DdR(V)\\
			\simeq & \left(\sum_{n \geq 0} \Fil^n \BdR \otimes_K \Fil^{-n} \DdR(V) \right) + \left(\sum_{n < 0} \BdR^+ \otimes_K \Fil^{-n} \DdR(V) \right).
		\end{split}
	\end{equation}
	We choose a splitting of the Hodge-Tate filtration of \(\DdR(V)\) so that
	\[
		\DdR(V) \simeq \bigoplus_{n \in \Z} \Gr^{-n} \DdR(V),
	\]
	where \(\Gr^{-n} \DdR(V) = \Fil^{-n} \DdR(V)/\Fil^{-n+1} \DdR(V)\).
	Then, the equation~\eqref{eq:cap} gives
	\begin{equation} \label{eq:oplus}
		\begin{split}
			& \BdR^+ \otimes_{\Qp} V \cap \BdR^+ \otimes_K \DdR(V)\\
			\simeq & \left( \bigoplus_{n \geq 0} \Fil^n \BdR \otimes_K \Gr^{-n} \DdR(V) \right) \oplus \left( \bigoplus_{n < 0} \BdR^+ \otimes_K \Gr^{-n} \DdR(V) \right).
		\end{split}
	\end{equation}
	By the equations~\eqref{eq:tvDdR} and~\eqref{eq:oplus}, we have
	\begin{equation} \label{eq:fracoplus}
		\begin{split}
			\h{t}_V(\Qpbar) & \simeq \frac{\BdR^+ \otimes_K \DdR(V)}{\BdR^+ \otimes_{\Qp} V \cap \BdR^+ \otimes_K \DdR(V)}\\
			& \simeq \bigoplus_{n \in \N} \frac{\BdR^+ \otimes_K \Gr^{-n} \DdR(V)}{\Fil^n \BdR \otimes_K \Gr^{-n} \DdR(V)} \oplus \bigoplus_{n < 0} \frac{\BdR^+ \otimes_K \Gr^{-n} \DdR(V)}{\BdR^+ \otimes_K \Gr^{-n} \DdR(V)}\\
			& \simeq \bigoplus_{n \in \N} \frac{\BdR^+ \otimes_K \Gr^{-n} \DdR(V)}{\Fil^n \BdR \otimes_K \Gr^{-n} \DdR(V)}\\
			& \simeq \bigoplus_{n \in \N} \B_n \otimes_K \Gr^{-n} \DdR(V).
		\end{split}
	\end{equation}
	Finally, since the multiplicity \(m_n(V)\) of \(n\) as a Hodge-Tate weight of \(V\) is the dimension \(\dim_K \Gr^{-n} \DdR(V)\), the equation~\eqref{eq:fracoplus} yields
	\[
		\h{t}_V(\Qpbar) \simeq \bigoplus_{n \in \N} \B_n^{m_n(V)}.
	\]
\end{proof}

If \(V\) is a de Rham \(p\)\=/adic representation of \(\GG_K\) and \(T\) is a \(\Zp\)\=/lattice in \(V\) stable under the action of \(\GG_K\), we have the sub\=/\(\GG_K\)\=/representation \(V_0\) from Definition~\ref{defi:V0}, and we set \(T_0 = T \cap V_0\) which is a \(\Zp\)\=/lattice in \(V_0\) stable under the action of \(\GG_K\).
\begin{coro} \label{coro:exactV0}
	Let \(V\) be a de Rham \(p\)\=/adic representation of \(\GG_K\).
	Let \(T\) be a \(\Zp\)\=/lattice in \(V\) stable under the action of \(\GG_K\).
	There is an exact commutative diagram
	\[
		\begin{tikzcd}
			& 0 \ar{d} & 0 \ar{d} &  & \\
			0 \ar{r} & V_0/T_0 \ar{r} \ar{d} & E_+(V_0/T_0) \ar{r} \ar{d} & \h{t}_{V_0}(\Qpbar) \ar{r} \ar{d}[sloped]{\sim} & 0\\
			0 \ar{r} & V/T \ar{r} \ar{d} & E_+(V/T) \ar{r} \ar{d} & \h{t}_{V}(\Qpbar) \ar{r} & 0\\
						& \frac{V/V_0}{T/T_0} \ar{r}{\sim} \ar{d} & E_+\left(\frac{V/V_0}{T/T_0}\right) \ar{d} &  & \\
			& 0  & 0 &  & .
		\end{tikzcd}
	\]
\end{coro}
\begin{proof}
	By Corollary~\ref{coro:HN}, we have the exact commutative diagram
	\begin{equation} \label{eq:truncV0}
		\begin{tikzcd}
			& 0 \ar{d} & 0 \ar{d} & 0 \ar{d} & \\
			0 \ar{r} & \EE(V_0) \ar{d} \ar{r}{\eta_{V_0}} & \EE_+(V_0) \ar{r} \ar{d} & \FF_+(V_0) \ar{r} \ar{d}[sloped]{\sim} & 0\\
			0 \ar{r} & \EE(V) \ar{r}{\eta_{V}} \ar{d} & \EE_+(V) \ar{r} \ar{d} & \FF_+(V) \ar{r} \ar{d} & 0\\
			0 \ar{r} & \EE(V/V_0) \ar{r}{\sim}[swap]{\eta_{V/V_0}} \ar{d} & \EE_+(V/V_0) \ar{r} \ar{d} & 0 & \\
			& 0 & 0 & & .
		\end{tikzcd}
	\end{equation}
	Once again, the Harder-Narasimhan slopes of all the coherent sheaves in the diagram~\eqref{eq:truncV0} are all greater than or equal to \(0\) (Lemma~\ref{lemm:HN0}), so by the isomorphisms established in Theorem~\ref{theo:RepBun0}, Proposition~\ref{prop:almostcoh} and Lemma~\ref{lemm:almostcoh}, the cohomology of the diagram~\eqref{eq:truncV0} yields the exact commutative diagram
	\begin{equation} \label{eq:truncV0II}
		\begin{tikzcd}
			& 0 \ar{d} & 0 \ar{d} &  & \\
			0 \ar{r} & V_0 \ar{r} \ar{d} & E_+(V_0) \ar{r} \ar{d} & \h{t}_{V_0}(\Qpbar) \ar{r} \ar{d}[sloped]{\sim} & 0\\
			0 \ar{r} & V \ar{r} \ar{d} & E_+(V) \ar{r} \ar{d} & \h{t}_{V}(\Qpbar) \ar{r} & 0\\
						& V/V_0 \ar{r}[sloped]{\sim} \ar{d} & E_+(V/V_0) \ar{d} &  & \\
			& 0  & 0 &  & .
		\end{tikzcd}
	\end{equation}
	Taking the quotient by \(T\) in the diagram~\eqref{eq:truncV0II}, we obtain the desired exact commutative diagram.
\end{proof}

\begin{rema} \label{rema:connectedetale}
	Let \(G\) be a \(p\)\=/divisible group defined over \(\OO_K\).
	Let \(G^\circ\) be the connected component of \(G\), and let \(G^\et=G/G^\circ\) be the maximal étale quotient of \(G\).
	The \(p\)\=/adic representation \(V_p(G)\) of \(\GG_K\) is de Rham with Hodge-Tate weights \(0\) and \(1\), and \(V_p(G^\et)=V_p(G)/V_p(G^\circ)\) is the maximal quotient of \(V_p(G)\) with unique Hodge-Tate weight \(0\).
	Thus, we have \[
		V_p(G)_0 = V_p(G^\circ),
	\]
	so \(G^\circ[p^\infty]=V_p(G)_0/T_p(G)_0\) and \(G^\et[p^\infty]=(V_p(G)/V_p(G)_0)/(T_p(G)/T_p(G)_0)\).
	Then, by Fontaine's result from Remark~\ref{rema:VarAbel}, the short exact sequence from Corollary~\ref{coro:exactV0} is isomorphic to the connected-étale short exact sequence~\cite[Proposition~4]{Tate67}
	\[
		\begin{tikzcd}
			0 \ar{r} & E_+(G^\circ[p^\infty]) \ar{r} \ar{d}[sloped]{\sim} & E_+(G[p^\infty]) \ar{r} \ar{d}[sloped]{\sim} & E_+(G^\et[p^\infty]) \ar{r} \ar{d}[sloped]{\sim} & 0\\
		0 \ar{r} & G^\circ(\OO_{\Cp}) \ar{r} & G(\OO_{\Cp}) \ar{r} & G^\et(\OO_{\Cp}) \ar{r} & 0.
		\end{tikzcd}
	\]

	Similarly, let \(A\) be an Abelian variety defined over \(K\).
	Then, the \(p\)\=/adic representation \(V_p(A)\) of \(\GG_K\) is de Rham with Hodge-Tate weights \(0\) and \(1\).
	Assume \(A\) has semi-stable reduction. Let \(\AA\) be a Néron model of \(A\), and let \(\h{\AA}\) be the commutative formal group associated with \(\AA\).
	Then, we have \(V_p(\h{\AA}) = V_p(A)_0 \subset V_p(A)\).
	Furthermore, if \(A\) has good reduction and \(\tilde{A}\) is its reduction over \(k_K\), then the short exact sequence from Corollary~\ref{coro:exactV0} is isomorphic to the reduction short exact sequence
	\[
		\begin{tikzcd}
			0 \ar{r} & E_+(\h{\AA}[p^\infty]) \ar{r} \ar{d}[sloped]{\sim} &
			E_+(A[p^\infty]) \ar{r} \ar{d}[sloped]{\sim} & E_+(\tilde{A}[p^\infty]) \ar{r} \ar{d}[sloped]{\sim} & 0\\
			0 \ar{r} & \h{\AA}(\mg_{\Cp}) \ar{r} & A^{(p)}(\Cp) \ar{r} & \tilde{A}(k_{\Cp})[p^\infty] \ar{r} & 0.
		\end{tikzcd}
	\]
\end{rema}

\subsection{Cohomology of perfectoid fields} \label{subsec:perfectoid}
Let \(L\) be an algebraic extension of \(K\).
If \(V\) is a de Rham \(p\)\=/adic representation of \(\GG_K\) and \(T\) is a \(\Zp\)\=/lattice in \(V\) stable under the action of \(\GG_K\), in this subsection we compute the cohomology group \(\H^1(L,E_+(V/T))\) when \(\h{L}\) is a perfectoid field.
We shall need some facts about Galois cohomology of perfectoid fields.
\begin{prop} \label{prop:cohomdim}
	If \(\h{L}\) is a perfectoid field, then the \(p\)\=/cohomological dimension of \(L\) is less than or equal to \(1\).
\end{prop}
\begin{proof}
	The tilt \(\h{L}^\flat\) of \(\h{L}\) is a field of characteristic \(p\) whose absolute Galois group is canonically isomorphic to \(\GG_L\) \cite[\S 3]{Scholze12}.
	Furthermore, the \(p\)\=/cohomological dimension of a field of characteristic \(p\) is less than or equal to \(1\) \cite[II \S 2.2 Proposition~3]{Serre94}.
\end{proof}

\begin{lemm} \label{lemm:tvhat}
	Let \(V\) be a \(p\)\=/adic representation of \(\GG_K\).
	If \(\h{L}\) is a perfectoid field, then
	\[
		\H^1(L,\h{t}_V(\Qpbar))=0.
	\]
\end{lemm}
\begin{proof}
	The \(\BdR^+\)\=/module underlying \(\h{t}_V(\Qpbar)\) is of finite length.
	Therefore, we can apply Proposition~\ref{prop:FFsplit} (and Remark~\ref{rema:FST}).
\end{proof}

\begin{lemm} \label{lemm:perfectoid}
	Let \(V\) be a \(p\)\=/adic representation of \(\GG_K\).
	Let \(T\) be a \(\Zp\)\=/lattice in \(V\) stable under the action of \(\GG_K\).
	If \(\h{L}\) is a perfectoid field, then the short exact sequence
	\[
		0 \ra T \ra E_+(V) \ra E_+(V/T) \ra 0
	\]
	induces a surjection
	\[
		\H^1(L,E_+(V)) \ra \H^1(L,E_+(V/T)) \ra 0.
	\]
\end{lemm}
\begin{proof}
	By Proposition~\ref{prop:cohomdim} and Lemma~\ref{lemm:tvhat}, the exact commutative diagram
	\[
		\begin{tikzcd}
			& 0 \ar{d} & 0 \ar{d} & & \\
			& T \ar[equal]{r} \ar{d} & T \ar{d} & & \\
			0 \ar{r} & V \ar{r} \ar{d} & E_+(V) \ar{r} \ar{d} & \h{t}_V(\Qpbar) \ar{r} \ar[equal]{d} & 0\\
			0 \ar{r} & V/T \ar{r} \ar{d} & E_+(V/T) \ar{r} \ar{d} & \h{t}_V(\Qpbar) \ar{r} & 0\\
			& 0 & 0 & &
		\end{tikzcd}
	\]
	induces the cohomological exact commutative diagram
	\[
		\begin{tikzcd}
			\H^1(L,V) \ar{r} \ar{d} & \H^1(L,E_+(V)) \ar{r} \ar{d} & 0 \\
			\H^1(L, V/T) \ar{r} \ar{d} & \H^1(L,E_+(V/T)) \ar{r} & 0\\
			0 & & ,
		\end{tikzcd}
	\]
	from which we deduce the surjectivity of the rightmost vertical map
	\[
		\H^1(L,E_+(V)) \ra \H^1(L,E_+(V/T)) \ra 0.
	\]
\end{proof}

\begin{prop} \label{prop:perfectoid}
	Let \(V\) be a de Rham \(p\)\=/adic representation of \(\GG_K\).
	If \(\h{L}\) is a perfectoid field, then
	\[
		\H^1(L,E_+(V_0)) = 0.
	\]
\end{prop}
\begin{proof}
	By Proposition~\ref{prop:almostcoh}, we have
	\[
		E_+(V_0) \simeq \Gamma(\XFF, \EE_+(V_0)),
	\]
	and the Harder-Narasimhan slopes of \(\EE_+(V_0)\) are strictly greater than \(0\) by Proposition~\ref{prop:HN0V0}.
	Therefore, by Proposition~\ref{prop:FFsplit}, we have
	\[
		\H^1(L,E_+(V_0)) \simeq \H^1(L,\Gamma(\XFF,\EE_+(V_0)) = 0.
	\]
\end{proof}

\begin{coro} \label{coro:perfectoidplus}
	Let \(V\) be a de Rham \(p\)\=/adic representation of \(\GG_K\).
	Let \(T\) be a \(\Zp\)\=/lattice in \(V\) stable under the action of \(\GG_K\).
	If \(\h{L}\) is a perfectoid field, then
	\begin{enumerate}
		\item \(\H^1(L,E_+(V_0/T_0)) = 0\),
		\item \(\H^1(L,E_+(V/T)) \simeq \H^1(L,(V/V_0)/(T/T_0))\).
	\end{enumerate}
\end{coro}
\begin{proof}
	The first assertion follows from Lemma~\ref{lemm:perfectoid} and Proposition~\ref{prop:perfectoid}.
	We prove the second one.
	By Corollary~\ref{coro:exactV0}, there is an exact commutative diagram
	\[
		\begin{tikzcd}
			& 0 \ar{d} & 0 \ar{d} &  & \\
			0 \ar{r} & V_0/T_0 \ar{r} \ar{d} & E_+(V_0/T_0) \ar{r} \ar{d} & \h{t}_{V_0}(\Qpbar) \ar{r} \ar{d}[sloped]{\sim} & 0\\
			0 \ar{r} & V/T \ar{r} \ar{d} & E_+(V/T) \ar{r} \ar{d} & \h{t}_{V}(\Qpbar) \ar{r} & 0\\
						& \frac{V/V_0}{T/T_0} \ar{r}{\sim} \ar{d} & E_+\left(\frac{V/V_0}{T/T_0}\right) \ar{d} &  & \\
			& 0  & 0 &  & ,
		\end{tikzcd}
	\]
	which, by Proposition~\ref{prop:cohomdim} and Lemma~\ref{lemm:tvhat}, induces the exact commutative diagram
	\begin{equation} \label{eq:diagV0}
		\begin{tikzcd}
			\H^1(L,V_0/T_0) \ar{r} \ar{d} & \H^1(L,E_+(V_0/T_0)) \ar{r} \ar{d} & 0\\
			\H^1(L,V/T) \ar{r} \ar{d} & \H^1(L,E_+(V/T)) \ar{r} \ar{d} & 0\\
		\H^1\left(L,\frac{V/V_0}{T/T_0}\right) \ar{r}{\sim} \ar{d} & \H^1\left(L,E_+\left(\frac{V/V_0}{T/T_0}\right)\right) &  & \\
			0  &  & .
		\end{tikzcd}
	\end{equation}
	The group \(\H^1(L,E_+(V_0/T_0))\) is trivial by the first point and we conclude using the commutativity of the diagram~\eqref{eq:diagV0}.
\end{proof}

\subsection{Universal norms} \label{subsec:universalnorms}
If \(V\) is a de Rham \(p\)\=/adic representation of \(\GG_K\), and \(T\) is a \(\Zp\)\=/lattice in \(V\) stable under the action of \(\GG_K\), and if \(L\) is an algebraic extension of \(K\), then the short exact sequence
\[
	0 \ra V_0/T_0 \ra V/T \ra (V/V_0)/(T/T_0) \ra 0
\]
induces the cohomological exact sequence
\[
	\H^1(L,V_0/T_0) \xra{\lambda_L} \H^1(L,V/T) \xra{\pi_L} \H^1(L,(V/V_0)/(T/T_0)).
\]

\begin{theo} \label{theo:principal}
	Let \(V\) be a de Rham \(p\)\=/adic representation of \(\GG_K\).
	Let \(T\) be a \(\Zp\)\=/lattice in \(V\) stable under the action of \(\GG_K\).
	Let \(L\) be an algebraic extension of \(K\).
	If \(\h{L}\) is a perfectoid field, then the exact commutative diagram
	\[
		\begin{tikzcd}
			0 \ar{r} & V/T \ar{r} & E_+(V/T) \ar{r} & \h{t}_V(\Qpbar) \ar{r} & 0\\
			0 \ar{r} & V/T \ar{r} \ar[equal]{u} & E_\discu(V/T) \ar{r} \ar{u} & t_V(\Qpbar) \ar{r} \ar{u} & 0
		\end{tikzcd}
	\]
	induces the exact commutative diagram
	\[
		\begin{tikzcd}
			0 \ar{r} & \Ker(\pi_L) \ar{r} & \H^1(L,V/T) \ar{r} & \H^1(L,E_+(V/T)) \ar{r} & 0\\
			0 \ar{r} & \H^1_e(L,V/T) \ar{r} \ar{u} & \H^1(L,V/T) \ar{r} \ar[equal]{u} & \H^1(L,E_\discu(V/T)) \ar{r} \ar{u} & 0.
		\end{tikzcd}
	\]
\end{theo}
\begin{proof}
	The theorem follows from the combination of Proposition~\ref{prop:BK}, Lemma~\ref{lemm:tvhat} and Corollary~\ref{coro:perfectoidplus}.
\end{proof}

\begin{coro} \label{coro:snake}
	Under the hypothesis of Theorem~\ref{theo:principal}, there is an isomorphism
	\[
		\Ker(\pi_L)/\H^1_e(L,V/T) \simeq \Ker\left(\H^1(L,E_\discu(V/T)) \ra \H^1(L,E_+(V/T))\right).
	\]
\end{coro}
\begin{proof}
The snake lemma applied to the diagram of Theorem~\ref{theo:principal} yields the desired isomorphism.
\end{proof}

We compute the right-hand side of the isomorphism from Corollary~\ref{coro:snake} under the strong additional assumption that the Hodge-Tate weights of the representation are less than or equal to \(1\).

\begin{theo} \label{theo:precis}
	Let \(V\) be a de Rham \(p\)\=/adic representation of \(\GG_K\).
	Let \(T\) be a \(\Zp\)\=/lattice in \(V\) stable under the action of \(\GG_K\).
	Let \(L\) be an algebraic extension of \(K\).
	Assume the Hodge-Tate weights of \(V\) are less than or equal to \(1\).
	If \(\h{L}\) is a perfectoid field, then there is a short exact sequence
	\[
		0 \ra \H^1_e(L,V/T) \ra \H^1(L,V/T) \xra{\pi_L} \H^1(L,(V/V_0)/(T/T_0)) \ra 0.
	\]
\end{theo}
\begin{proof}
	Proposition~\ref{prop:cohomdim} implies the surjectivity of \(\pi_L\).
	We prove that \(\H^1_e(L,V/T) = \Ker(\pi_L)\).
	By Corollary~\ref{coro:Bn}, the hypothesis on the Hodge-Tate weights of \(V\) implies that
	\[
		\h{t}_V(\Qpbar) = t_V(\Cp).
	\]
	Thus, by Ax-Sen-Tate theorem~\cite{Tate67}, we have
	\begin{equation} \label{eq:H0tvCp}
		\H^0(L,\h{t}_V(\Qpbar)) = \H^0(L,t_V(\Cp)) = t_V(\h{L}).
	\end{equation}
	By Theorem~\ref{theo:principal} and the equation~\eqref{eq:H0tvCp}, there is an exact commutative diagram of Abelian groups
	\begin{equation} \label{eq:continuity}
		\begin{tikzcd}
			0 \ar{r} & (V/T)^{\GG_L} \ar{r} & E_+(V/T)^{\GG_L} \ar{r} & t_V(\h{L}) \ar{r} & \Ker(\pi_L) \ar{r} & 0\\
			0 \ar{r} & (V/T)^{\GG_L} \ar{r} \ar[equal]{u} & E_\discu(V/T)^{\GG_L} \ar{r} \ar{u} & t_V(L) \ar{r} \ar{u} & \H^1_e(L,V/T) \ar{r} \ar{u} & 0.
		\end{tikzcd}
	\end{equation}
	The diagram~\eqref{eq:continuity} can also be considered as a diagram of topological Abelian groups where \(\H^1_e(L,V/T) \subset \Ker(\pi_L) \subset \H^1(L,V/T)\) are discrete topological groups (see Appendix~\ref{apdx:contcohom}).
	Furthermore, \(t_V(L)\) is dense in \(t_V(\h{L})\).
	Therefore, by continuity, the diagram~\eqref{eq:continuity} yields
	\[
		\H^1_e(L,V/T) = \Ker(\pi_L).
	\]
\end{proof}

\begin{rema}
	While \(\Qpbar\) is dense in \(\BdR^+\) \cite{Colmez12}, if \(\h{L}\) is a perfectoid field, then \(L\) is in general not dense in \(\B_n^{\GG_L}\) for \(n > 1\) \cite{IovitaZaharescu99}, so the proof of Theorem~\ref{theo:precis} can not be extended to higher Hodge-Tate weights.
\end{rema}

\begin{rema} \label{rema:CGcomparaison}
	Let \(A\) be an Abelian variety defined over \(K\).
	We noted in Remarks~\ref{rema:Kummer} and~\ref{rema:connectedetale} that the representation \(V_p(A)\) is de Rham with Hodge-Tate weights \(0\) and \(1\), and that the exponential group coincides with the image of the Kummer map.
	Therefore, Theorem~\ref{theo:precis} applies to \(V_p(A)\), and we to recover Coates and Greenberg's result for Abelian varieties mentioned in the introduction (equation~\eqref{eq:CGtheo_intro}).
	Coates and Greenberg's original statement~\cite[Proposition~4.3]{CoatesGreenberg96} differs from the equivalent formulation given in the introduction on the following points.
	\begin{enumerate}
		\item Coates and Greenberg use the notion of \emph{deeply ramified extension} to state their result (notion which they also introduced), but \(L\) is a deeply ramified extension of \(K\) if and only if \(\h{L}\) is a perfectoid field~\cite[Remark~3.3]{Scholze12}.
		\item Instead of \(V_p(A)_0\), Coates and Greenberg use the representation \(V_p(A)^\prime\) defined as the minimal sub\=/\(\GG_K\)\=/representation of \(V_p(A)\) such that the inertia subgroup \(\II_K \subset \GG_K\) acts through a finite quotient on \(V_p(A)/V_p(A)^\prime\).
		We have \(V_p(A)_0 = V_p(A)^\prime\).
		Indeed, \(V_p(A)/V_p(A)_0\) is the maximal quotient of \(V_p(A)\) with unique Hodge-Tate weight \(0\), and Sen~\cite{Sen73} proved that if \(W\) is a de Rham \(p\)\=/adic representation of \(\GG_K\), then the inertia subgroup \(\II_K\) acts through a finite quotient on \(W\) if and only if \(0\) is the unique Hodge-Tate weight of \(W\).
	\end{enumerate}
\end{rema}

As a by-product of Theorem~\ref{theo:precis} and Corollary~\ref{coro:snake}, we obtain:
\begin{coro} \label{coro:EdiscequalEplus}
	Let \(V\) be a de Rham \(p\)\=/adic representation of \(\GG_K\).
	Let \(T\) be a \(\Zp\)\=/lattice in \(V\) stable under the action of \(\GG_K\).
	Let \(L\) be an algebraic extension of \(K\).
	Assume the Hodge-Tate weights of \(V\) are less than or equal to \(1\).
	If \(\h{L}\) is a perfectoid field, then
	\[
		\H^1(L,E_\discu(V/T)) \simeq \H^1(L,E_+(V/T)).
	\]
\end{coro}

The combination of Corollaries~\ref{coro:perfectoidplus} and~\ref{coro:EdiscequalEplus} yields:
\begin{coro} \label{coro:EdiscCG}
	Let \(V\) be a de Rham \(p\)\=/adic representation of \(\GG_K\).
	Let \(T\) be a \(\Zp\)\=/lattice in \(V\) stable under the action of \(\GG_K\).
	Let \(L\) be an algebraic extension of \(K\).
	Assume the Hodge-Tate weights of \(V\) are less than or equal to \(1\).
	If \(\h{L}\) is a perfectoid field, then
	\[
		\begin{split}
			\H^1(L,E_\discu(V_0/T_0)) & = 0\\
			\H^1(L,E_\discu(V/T)) & \simeq \H^1(L,(V/V_0)/(T/T_0)).
		\end{split}
	\]
\end{coro}

\begin{rema} \label{rema:byproductEdisc}
	Let \(G\) is a \(p\)\=/divisible group defined over \(\OO_K\), and let \(L\) be an algebraic extension of \(K\).
	By Remarks~\ref{rema:VarAbel} and~\ref{rema:connectedetale}, we have
	\[
		E_\discu(V_p(G)/T_p(G)) \simeq G(\OO_{\Qpbar}), \text{ and } E_\discu(V_p(G)_0/T_p(G)_0) \simeq G^\circ(\OO_{\Qpbar}).
	\]
	Therefore, by Corollary~\ref{coro:EdiscCG},	if \(\h{L}\) is perfectoid, then
	\[
			\H^1(L,G^\circ) = 0, \text{ and } \H^1(L,G) \simeq \H^1(L,G^\et).
	\]
	Hence, we recover a theorem of Coates and Greenberg \cite[Theorem~3.1, Corollary~3.2]{CoatesGreenberg96} who computed these cohomology groups using explicit methods.
	Similarly, we recover the following result of Coates and Greenberg~\cite[Proposition~4.8]{CoatesGreenberg96} for Abelian varieties.
	Let \(A\) be an Abelian variety defined over \(K\).
	By Corollary~\ref{coro:EdiscCG}, if \(\h{L}\) is perfectoid, then
	\[
		\H^1(L,A^{(p)}) \simeq \H^1(L,(V_p(A)/V_p(A)_0)/(T_p(A)/T_p(A)_0)),
	\]
	and we recall that if \(A\) has good reduction and \(\tilde{A}\) is its reduction over \(k_K\), then
	\[
		\tilde{A}(k_{\Qpbar})[p^\infty] \simeq (V_p(A)/V_p(A)_0)/(T_p(A)/T_p(A)_0).
	\]
\end{rema}

If \(V\) is a \(p\)\=/adic representation of \(\GG_K\), and \(T\) a \(\Zp\)\=/lattice in \(V\) stable under the action of \(\GG_K\), then, for each \(\ast \in \{e,f,g\}\), the groups \(\H^1_\ast(K^\prime,T)\), where \(K^\prime\) is a finite extension of \(K\), are compatible under the corestriction maps.
If \(L\) is an algebraic extension of \(K\), the first \emph{Iwasawa cohomology group} of the extension \(L/K\) with coefficients in \(T\) is
\[
	\H^1_\Iw(L/K,T) = \limproj_{\cor, K^\prime} \H^1(K^\prime,T),
\]
and, for \(\ast \in \{e,f,g\}\), the module of \emph{\(\ast\)\=/universal norms} associated with \(T\) in the extension \(L/K\) is
\[
	\H^1_{\Iw, \ast}(L/K,T) = \limproj_{\cor, K^\prime} \H^1_\ast(K^\prime,T),
\]
where both limits are taken relatively to the corestriction maps and \(K^\prime\) runs through the finite extensions of \(K\) contained in \(L\).
If \(L\) is finite over \(K\), then \(\H^1_\Iw(L/K,T) = \H^1(L,T)\) and \(\H^1_{\Iw, \ast}(L/K,T)=\H^1_\ast(L,T)\).

\begin{rema}
	If \(A\) is an Abelian variety defined over \(K\), then
	\[
		\H^1_{\Iw, g}(L/K,T_p(A)) \simeq \limproj_{\norm, K^\prime} A^{(p)}(K^\prime)
	\]
	where the limit on the right-hand side is taken relatively to the norm maps and \(K^\prime\) runs through the finite extensions of \(K\) contained in \(L\).
\end{rema}

If \(V\) is a \(p\)\=/adic representation of \(\GG_K\), we set \(V^\ast(1) = \Hom_{\Qp}(V,\Qp(1))\) the Tate dual of \(V\).
If \(T\) is a \(\Zp\)\=/lattice in \(V\) stable under the action of \(\GG_K\), we set \(T^\ast(1) = \Hom_{\Zp}(T,\Zp(1))\) the Tate dual of \(T\).
For each finite extension \(K^\prime\) of \(K\), we recall~\cite[Proposition~3.8]{BlochKato90} that under Tate local duality
\[
	\H^1(K^\prime, V/T) \times \H^1(K^\prime, T^\ast(1)) \ra \H^2(K^\prime,
	\Qp(1)/\Zp(1)) \simeq \Qp/\Zp,
\]
we have:
\begin{itemize}
	\item \(\H^1_e(K^\prime,V/T)\) is the orthogonal complement of \(\H^1_g(K^\prime,T^\ast(1))\),
	\item \(\H^1_f(K^\prime,V/T)\) is the orthogonal complement of \(\H^1_f(K^\prime,T^\ast(1))\),
	\item \(\H^1_g(K^\prime,V/T)\) is the orthogonal complement of \(\H^1_e(K^\prime,T^\ast(1))\).
\end{itemize}
If \(L\) is an algebraic extension of \(K\), then Tate local duality induces a perfect pairing
\[
	\H^1(L, V/T) \times \H^1_\Iw(L/K, T^\ast(1)) \ra \Qp/\Zp,
\]
under which:
\begin{itemize}
	\item \(\H^1_e(L,V/T)\) is the orthogonal complement of \(\H^1_{\Iw,g}(L/K,T^\ast(1))\),
	\item \(\H^1_f(L,V/T)\) is the orthogonal complement of \(\H^1_{\Iw,f}(L/K,T^\ast(1))\),
	\item \(\H^1_g(L,V/T)\) is the orthogonal complement of \(\H^1_{\Iw,e}(L/K,T^\ast(1))\).
\end{itemize}

In particular, we have the following corollary of Proposition~\ref{prop:BK}.
\begin{coro} \label{coro:Iwasawa}
	Let \(V\) be a \(p\)\=/adic representation of \(\GG_K\).
	Let \(T\) be a \(\Zp\)\=/lattice in \(V\) stable under the action of \(\GG_K\).
	Let \(L\) be an algebraic extension of \(K\).
	Tate local duality induces a perfect pairing
	\[
		\H^1(L,E_\discu(V/T)) \times \H^1_{\Iw,g}(L/K,T^\ast(1)) \ra \Qp/\Zp.
	\]
\end{coro}

If \(V\) is a de Rham \(p\)\=/adic representation of \(\GG_K\), we denote by \(\Fil^1 V\) the maximal sub\=/\(\GG_K\)\=/representation of \(V\) whose Hodge-Tate weights are greater than or equal to \(1\).
If \(T\) is a \(\Zp\)\=/lattice in \(V\) stable under the action of \(\GG_K\), we set \(\Fil^1 T = \Fil^1 V \cap T\).
We have the relation
\[
	(\Fil^1 V)^\ast(1) = V^\ast(1)/(V^\ast(1))_0.
\]
Therefore, under Tate local duality, the dual statement of Theorem~\ref{theo:principal} is the following.
\begin{theo} \label{theo:Iwasawa}
	Let \(V\) be a de Rham \(p\)\=/adic representation of \(\GG_K\).
	Let \(T\) be a \(\Zp\)\=/lattice in \(V\) stable under the action of \(\GG_K\).
	Let \(L\) be an algebraic extension of \(K\).
	Assume the Hodge-Tate weights of \(V\) are greater than or equal to \(0\).
	If \(\h{L}\) is a perfectoid field, then
	\[
		\H^1_{\Iw,g}(L/K,T) \simeq \H^1_\Iw(L/K,\Fil^1 T).
	\]
\end{theo}

\begin{rema} \label{rema:Sen}
	By a theorem of Sen~\cite{Sen72} (see also \cite[Theorem~2.13]{CoatesGreenberg96}), if \(L\) is an infinite Galois extension of \(K\) whose Galois group \(\Gal(L/K)\) is a \(p\)\=/adic Lie group and such that \(L_0\) is a finite extension over \(\Qp\), then \(\h{L}\) is a perfectoid field.
	Therefore, the previous results apply to such extensions which are frequently encountered in non-commutative Iwasawa theory~\cite{CFKSV05}.
	In particular, if \(L = \bigcup_{n \in \N} K(\zeta_{p^n})\) is the cyclotomic extension, then \(\h{L}\) is perfectoid.
	We note some relations of our results with the literature:
	\begin{enumerate}
		\item If \(L\) is the cyclotomic extension of \(K\), and if \(V\) is a de Rham representation of \(\GG_K\), and \(T\) a \(\Zp\)\=/lattice in \(V\) stable under the action of \(\GG_K\), Berger~\cite{Berger05} proved that the quotient \(\H^1_\Iw(L/K,\Fil^1 T)/\H^1_{\Iw,g}(L/K,T)\) is a torsion module over the Iwasawa algebra \(\Zp[[\Gal(L/K)]]\).
			In particular, Berger's theorem requires no restriction on the Hodge-Tate weights of \(V\).
			However, the author does not know how to recover Berger's result from Theorem~\ref{theo:principal}.
			To do so would be an interesting first step to remove the hypothesis on the Hodge-Tate weights in Theorems~\ref{theo:precis} and~\ref{theo:Iwasawa} (and Theorems~\ref{theo:Iwasawafini} and~\ref{theo:Iwasawa_autresBK} below).
		\item Again in the cyclotomic case, the module \((\Be \otimes_{\Qp} V)^{\GG_L} = E_\e(V)^{\GG_L}\) appears in Colmez's work~\cite{Colmez98} on Iwasawa theory, where it is denoted by \(\D_\Iw(V)\).
		\item If \(L\) is the localization at a prime dividing \(p\) of the \(\Zp\)\=/anticyclotomic extension of an imaginary quadratic field in which \(p\) is inert, then by class field theory (see for example~\cite[\S 3]{Serre58}), Sen's theorem applies and thus \(\h{L}\) is a perfectoid field.
			Theorem~\ref{theo:Iwasawa} then partially answers a series of conjectures formulated by Büyükboduk~\cite[Conjectures~2.5, 2.6 and 2.7]{Buyukboduk15}.
	\end{enumerate}
\end{rema}

If \(V\) is a \(p\)\=/adic representation of \(\GG_K\) and \(T\) is a \(\Zp\)\=/lattice in \(V\) stable under the action of \(\GG_K\), for \(\ast \in \{e,f,g\}\), we set
\[
		\H^1_{\ast,L}(K,T)   = \bigcap_{K^\prime \subset L} \cor_{K^\prime/K} \left(\H^1_\ast(K^\prime,T)\right),
\]
where \(K^\prime\) runs through the finite extension of \(K\) contained in \(L\).
Then \(\H^1_{\ast,L}(K,T)\) is also the image of the natural projection map \(\H^1_{\Iw,\ast}(L/K,T) \ra \H^1(K,T)\).

We shall now use Theorem~\ref{theo:precis} and its corollaries to give a description of \(\H^1_{g,L}(K,T)\) which generalises a theorem of Coates and Greenberg for Abelian varieties~\cite[Theorem~5.2]{CoatesGreenberg96}.
We need some notations.
Let
\[
	\res_{L/K}^0 : \H^1(K,(V/V_0)/(T/T_0)) \ra \H^1(L,(V/V_0)/(T/T_0))
\]
be the restriction map.
We have the previous map
\[
	\pi_K: \H^1(K,V/T) \ra \H^1(K,(V/V_0)/(T/T_0)),
\]
and we set
\[
	\Omega(L/K,V/T) = \frac{\Img(\pi_K)}{\Img(\pi_K)\cap \Ker(\res_{L/K}^0)}.
\]
\begin{theo} \label{theo:Iwasawafini}
	Let \(V\) be a de Rham \(p\)\=/adic representation of \(\GG_K\).
	Let \(T\) be a \(\Zp\)\=/lattice in \(V\) stable under the action of \(\GG_K\).
	Let \(L\) be an algebraic extension of \(K\).
	Assume the Hodge-Tate weights of \(V\) are less than or equal to \(1\).
	If \(\h{L}\) is a perfectoid field, then Tate local duality induces a perfect pairing
	\[
		\H^1_{g,L}(K,T^\ast(1)) \times \Omega(L/K,V/T) \ra \Qp/\Zp.
	\]
\end{theo}
\begin{proof}
	We follow closely Coates and Greenberg's proof~\cite[Theorem~5.2]{CoatesGreenberg96}.
	Let
	\[
		\res_{L/K}^\disc : \H^1(K,E_\discu(V/T)) \ra \H^1(L,E_\discu(V/T))
	\]
	be the restriction map.
	By Corollary~\ref{coro:Iwasawa}, Tate local duality induces a perfect pairing
	\[
		\Img(\res_{L/K}^\disc) \times \H^1_{g,L}(K,T^\ast(1)) \ra \Qp/\Zp.
	\]
	Therefore, it is enough to prove that there is a canonical isomorphism
	\begin{equation} \label{eq:desiredisoResOmega}
		\Img(\res_{L/K}^\disc) \simeq \Omega(L/K,V/T).
	\end{equation}
	By Proposition~\ref{prop:BK}, there is an exact commutative diagram
	\[
		\begin{tikzcd}
			\H^1(L,V/T) \ar{r}{\pi_L^\disc} & \H^1(L,E_\discu(V/T)) \ar{r} & 0\\
			\H^1(K,V/T) \ar{r}{\pi_K^\disc} \ar{u}{\res_{L/K}} & \H^1(K,E_\discu(V/T)) \ar{u}{\res_{L/K}^\disc} \ar{r} & 0,
		\end{tikzcd}
	\]
	thus, we have
	\begin{equation} \label{eq:pidisc}
		\Img(\res_{L/K}^\disc) = \Img(\res_{L/K}^\disc \circ \pi_K^\disc) = \Img(\pi_L^\disc \circ \res_{L/K}).
	\end{equation}
	By Corollary~\ref{coro:EdiscCG}, we have the commutative diagram
	\[
		\begin{tikzcd}
			\H^1(L,V/T) \ar{r}{\pi_L^\disc} \ar[equal]{d} & \H^1(L,E_\discu(V/T)) \ar{d}[sloped]{\sim} \\
			\H^1(L,V/T) \ar{r}{\pi_L} & \H^1(L,(V/V_0)/(T/T_0)),
		\end{tikzcd}
	\]
	hence, we have
	\begin{equation} \label{eq:pidiscpi0}
		\Img(\pi_L^\disc \circ \res_{L/K}) = \Img(\pi_L \circ \res_{L/K}).
	\end{equation}
	Finally, the commutative diagram
	\[
		\begin{tikzcd}
			\H^1(L,V/T) \ar{r}{\pi_L} & \H^1(L,(V/V_0)/(T/T_0))\\
			\H^1(K,V/T) \ar{r}{\pi_K} \ar{u}{\res_{L/K}} & \H^1(K,(V/V_0)/(T/T_0)) \ar{u}{\res_{L/K}^0}
		\end{tikzcd}
	\]
	gives
	\begin{equation} \label{eq:pi0res0}
		\Img(\pi_L \circ \res_{L/K}) = \Img(\res^0_{L/K} \circ \pi_K) \simeq \Omega(L/K,V/T).
	\end{equation}
	The combination of the equations~\eqref{eq:pidisc}, \eqref{eq:pidiscpi0} and \eqref{eq:pi0res0} yields the desired isomorphism~\eqref{eq:desiredisoResOmega}.
\end{proof}

\begin{rema} \label{rema:otherBK}
	It is possible to extend each result of this subsection to the other Bloch-Kato subgroups under additional assumptions.
	Let \(V\) be a \(p\)\=/adic representation of \(\GG_K\).
	If \(K^\prime\) is a finite extension of \(K\), then
	\[
		\D_{\cris, K^\prime}(V) = (\Bcris \otimes_{\Qp} V)^{\GG_{K^\prime}}
	\]
	is a finite-dimensional \(K^\prime_0\)\=/vector space equipped with the semi-linear endomorphism \(\phi \otimes 1\), simply denoted by \(\phi\), and we have \cite[Corollary~3.8.4]{BlochKato90}
	\[
		\begin{split}
			\H^1_f(K^\prime,V)/\H^1_e(K^\prime,V) & \simeq \D_{\cris, K^\prime}(V)/(1- \phi)\D_{\cris, K^\prime}(V),\\
			\H^1_g(K^\prime,V)/\H^1_f(K^\prime,V) & \simeq \D_{\cris, K^\prime}(V)^{\phi = p^{-1}}.
		\end{split}
	\]
	Therefore, we have the two following ways to control the Bloch-Kato subgroups from one another.
	\begin{enumerate}
		\item If \(\D_{\cris, K^\prime}(V)/(1- \phi)\D_{\cris, K^\prime}(V)\) (respectively \(\D_{\cris, K^\prime}(V)^{\phi = p^{-1}}\)) is trivial, then
		\[
			\H^1_e(K^\prime,V) = \H^1_f(K^\prime,V), \quad (\text{respectively }\H^1_f(K^\prime,V) = \H^1_g(K^\prime,V)).
		\]
		\item If \(L\) is an algebraic extension of \(K\) such that \(L_0\) is a finite extension over \(\Qp\), then the dimension
		\[
			\dim_{\Qp}\left( \H^1_g(K^\prime,V)/\H^1_e(K^\prime,V) \right) \leq [K^\prime_0:\Qp] \cdot \dim_{K^\prime_0} \D_{\cris,K^\prime}(V)
		\]
		is bounded as \(K^\prime\) varies through the finite extensions of \(K\) contained in \(L\).
	\end{enumerate}
\end{rema}

In particular, we obtain:
\begin{theo} \label{theo:Iwasawa_autresBK}
	Let \(V\) be a de Rham \(p\)\=/adic representation of \(\GG_K\).
	Let \(T\) be a \(\Zp\)\=/lattice in \(V\) stable under the action of \(\GG_K\).
	Let \(L\) be an infinite Galois extension of \(K\) whose Galois group is a \(p\)\=/adic Lie group and such that \(L_0\) is finite over \(\Qp\).
	Assume the Hodge-Tate weights of \(V\) are greater than or equal to \(0\).
	Then, for \(\ast \in \{e,f\}\), the quotient
	\[
		\H^1_\Iw(L/K,\Fil^1 T)/\H^1_{\Iw,\ast}(L/K,T)
	\]
	is of finite \(\Zp\)\=/rank.
\end{theo}
\begin{proof}
	By Sen's theorem (Remark~\ref{rema:Sen}), the field \(\h{L}\) is perfectoid, so Theorem~\ref{theo:Iwasawa} applies and we have
	\[
		\H^1_{\Iw,g}(L/K,T) \simeq \H^1_\Iw(L/K,\Fil^1 T).
	\]
	Furthermore, the last criterion of Remark~\ref{rema:otherBK} applies to the extension \(L/K\) and implies that, for \(\ast \in \{e,f\}\), the quotient of
	\[
		\H^1_{\Iw,\ast}(L/K,T) \subset \H^1_{\Iw,g}(L/K,T)
	\]
	is of finite \(\Zp\)\=/rank.
\end{proof}

\appendix
\section{Continuous group cohomology and the compact\texorpdfstring{\-/}{-}open topology} \label{apdx:contcohom}
We recall the definition of continuous group cohomology~\cite[\S 2]{Tate76}.
We shall note that for locally compact and separated topological groups, the continuous group cohomology groups can be equipped with a well-behaved structure of topological group induced by the compact-open topology on the sets of continuous cochains\footnote{This observation is certainly not new, but the author was not able to find a reference for it.}.

\subsection{The compact-open topology}
If \(X\) and \(Y\) are topological spaces, then we denote by \(\CC(X,Y)\) the set of continuous maps from \(X\) to \(Y\), and by \(\CC_c(X,Y)\) the topological space obtained by endowing \(\CC(X,Y)\) with the compact-open topology generated by the subsets
\[
	T(K,U) = \{ f \in \CC(X,Y), f(K) \subset U \},
\]
where \(K\) runs through the compact subsets of \(X\) and \(U\) through the open subsets of \(Y\).

\begin{lemm} \label{lemm:functorialityCc}
	Let \(X\), \(X^\prime\), \(Y\) and \(Y^\prime\) be topological spaces.
	Let \(\phi: X \ra X^\prime\) and \(\psi: Y \ra Y^\prime\) be continuous maps.
	The map
	\[
		\begin{split}
			\pi: \CC_c(X^\prime,Y) & \ra \CC_c(X,Y^\prime)\\
			f & \mapsto \psi \circ f \circ \phi
		\end{split}
	\]
	is continuous.
\end{lemm}
\begin{proof}
	Let \(K\) be a compact subset of \(X\), and let \(U^\prime\) be an open subset of \(Y^\prime\).
	We have \(\pi^{-1}(T(K,U^\prime))=T(\phi(K),\psi^{-1}(U^\prime))\).
	Furthermore, by continuity of \(\phi\) and \(\psi\) respectively, the set \(\phi(K)\) is compact in \(X^\prime\), and \(\psi^{-1}(U^\prime)\) is open in \(Y\).
	Therefore, the set \(\pi^{-1}(T(K,U^\prime))\) is open in \(\CC_c(X^\prime,Y)\).
\end{proof}

Let \(\Top\) be the category of topological spaces.
By Lemma~\ref{lemm:functorialityCc}, there is a bifunctor
\[
	\CC_c(\cdot, \cdot): \Top^\op \times \Top \ra \Top.
\]
If \(X\) is a topological space, we set 
\[
	\begin{split}
		X \times (\cdot):  \Top & \ra \Top\\
		Y & \mapsto X \times Y.
	\end{split}
\]
the product functor, where:
\begin{itemize}
	\item \(X\times Y\) is equipped with product topology for any topological space \(Y\), and
	\item a continuous map \(\psi: Y \ra Y^\prime\) is mapped to \((\id_X, \psi): X \times Y \ra X \times Y^\prime\).
\end{itemize}

\begin{theo} \label{theo:compactopenadjoint}
	If \(X\) is a locally compact and separated topological space, then the functor
	\[
		\CC_c(X, \cdot):  \Top \ra \Top
	\]
	is right adjoint to the product functor \(X \times (\cdot)\).
\end{theo}
\begin{proof}
	Let \(Y\) and \(Z\) be topological spaces.
	If \(X\) is a locally compact and separated topological space, then the map
	\[
		\begin{split}
			\CC(X\times Y, Z) & \ra \CC(Y,\CC_c(X,Z))\\
			f & \mapsto [y \mapsto f(\cdot , y)]
		\end{split}
	\]
	is a bijection \cite[X \S 3 no.~4 Théorème~3]{BourbakiTG5_10}\footnote{A \emph{localement compact} topological space as defined by Bourbaki~\cite[I \S 9 no.~7 Définition~4]{BourbakiTG1_4} is a locally compact and separated topological space.}, which is obviously functorial in \(Y\) and \(Z\). 
\end{proof}

\begin{coro} \label{coro:compactopengroup}
	Let \(X\) be a locally compact and separated topological space.
	If \(M\) is a topological group, then the topological space \(\CC_c(X,M)\) endowed with the group structure induced by \(M\) is a topological group.
\end{coro}
\begin{proof}
	The functor \(\CC_c(X,\cdot)\) is a right adjoint by Theorem~\ref{theo:compactopenadjoint}, and therefore it is continuous~\cite[V \S 5 Theorem~1]{MacLane98}.
	Hence, it associates group objects with group objects.
\end{proof}

\begin{coro} \label{coro:evaluation}
	Let \(X\) be a locally compact and separated topological space.
	Let \(Y\) be a topological space.
	The \emph{evaluation} map
	\[
		\begin{split}
			X \times \CC_c(X,Y) & \ra Y\\
			(x, f) & \mapsto f(x)
		\end{split}
	\]
	is continuous and is the universal morphism from the product functor \(X\times (\cdot)\) to \(Y\).
\end{coro}
\begin{proof}
	By Theorem~\ref{theo:compactopenadjoint}, there is a bijection
	\[
		\begin{split}
			\CC(\CC_c(X,Y),\CC_c(X,Y)) & \ra \CC(X\times \CC_c(X,Y), Y)\\
			\phi & \mapsto [(x,f) \mapsto \phi(f)(x)]
		\end{split}
	\]
	under which the identity map \(\id_{\CC_c(X,Y)} \in \CC(\CC_c(X,Y),\CC_c(X,Y))\) is mapped to the evaluation map
	\[
		\begin{split}
			X \times \CC_c(X,Y) & \ra Y\\
			(x, f) & \mapsto f(x),
		\end{split}
	\]
	which is therefore continuous and is the universal morphism from the product functor \(X\times (\cdot)\) to \(Y\).
\end{proof}

\begin{coro} \label{coro:coboundpart}
	Let \(X\) and \(X^\prime\) be locally compact and separated topological spaces.
	Let \(Y\) and \(Z\) be topological spaces, and let
	\[
			\psi: X \times Y \ra Z
	\]
	be a continuous map.
	The map
	\[
		\begin{split}
			\CC_c(X^\prime,Y) & \ra \CC_c( X \times X^\prime , Z)\\
			f & \mapsto [ (x,x^\prime) \mapsto \psi(x,f(x^\prime))]
		\end{split}
	\]
	is continuous.
\end{coro}
\begin{proof}
	By Corollary~\ref{coro:evaluation}, the composition of \(\psi\) with the evaluation map
	\begin{equation} \label{eq:actioneval}
		\begin{array}{r@{\ }c@{\ }c@{\ }c@{\ }l}
			X \times X^\prime \times \CC_c(X^\prime,Y) & \ra & X \times Y & \ra & Z\\
			(x,x^\prime,f) & \mapsto & (x,f(x^\prime)) & \mapsto & \psi(x , f(x^\prime))
		\end{array}
	\end{equation}
	is continuous.
	Moreover, Theorem~\ref{theo:compactopenadjoint} applied to \(X \times X^\prime\) yields a bijection 
	\[
		\begin{split}
			\CC(X \times X^\prime \times \CC_c(X^\prime,Y) ,Z) & \ra \CC(\CC_c(X^\prime,Y), \CC_c(X \times X^\prime , Z))\\
			\xi & \mapsto [f \mapsto \xi(\cdot, \cdot, f)]
		\end{split}
	\]
	under which the continuous application~\eqref{eq:actioneval} is mapped to
	\[
		\begin{split}
			\CC_c(X^\prime,Y) & \ra \CC_c(X \times X^\prime , Z)\\
			f & \mapsto [ (x,x^\prime) \mapsto \psi(x,f(x^\prime))]
		\end{split}
	\]
	 which is therefore continuous.
\end{proof}

We shall also need the following lemma.
\begin{lemm} \label{lemm:relativetop}
	Let \(X\), \(Y\) and \(Y^\prime\) be topological spaces.
	Let \(\psi: Y \ra Y^\prime\) be a continuous map, and let
	\[
		\begin{split}
			\pi: \CC_c(X,Y) & \ra \CC_c(X,Y^\prime)\\
			f & \mapsto \psi \circ f,
		\end{split}
	\]
	be the continuous map induced by \(\psi\).
	\begin{enumerate}
		\item If \(\psi\) is injective, then \(\pi\) is injective.
			Moreover, if the topology of \(Y\) is the subspace topology from \(Y^\prime\), then the topology of \(\CC_c(X,Y)\) is the subspace topology from \(\CC_c(X,Y^\prime)\).
	\item If \(\psi\) is surjective and admits a continuous section \(s:Y^\prime \ra Y\), then \(\pi\) is surjective, and the topology of \(\CC_c(X,Y^\prime)\) is the quotient topology from \(\CC_c(X,Y)\).
	\end{enumerate}
\end{lemm}
\begin{proof}
	We prove the first statement.
	If \(\psi\) is injective, the injectivity of \(\pi\) is obvious.
	We now further assume that the topology of \(Y\) is the subspace topology from \(Y^\prime\).
	Let \(K\) be a compact subset of \(X\), and let \(U\) be an open subset of \(Y\).
	Since the topology of \(Y\) is the subspace topology from \(Y^\prime\), there exists an open subset \(U^\prime\) of \(Y^\prime\) such that \(U = \psi^{-1}(U^\prime)\).
	Therefore, we have
	\[
		T(K,U) = T(K,\psi^{-1}(U^\prime)) = \pi^{-1}(T(K,U^\prime))
	\]
	in \(\CC_c(X,Y)\).
	Hence, the topology of \(\CC_c(X,Y)\) is the subspace topology from \(\CC_c(X,Y^\prime)\).

	We prove the second statement.
	If \(\psi\) is surjective and admits a continuous section \(s:Y^\prime \ra Y\), then \(s\) induces a continuous map
	\[
		\begin{split}
			\CC_c(X,Y^\prime) & \ra \CC_c(X,Y)\\
			f & \mapsto s \circ f
		\end{split}
	\]
	which is a continuous section of \(\pi\).
	Therefore, the map \(\pi\) is surjective, and the topology of \(\CC_c(X,Y^\prime)\) is the quotient topology from \(\CC_c(X,Y)\) \cite[I \S 3 no.~5 Proposition~9]{BourbakiTG1_4}.
\end{proof}

\subsection{Continuous group cohomology}
We refer the reader to \cite{BourbakiA10} for the classical results of homological algebra reviewed in the following paragraphs.

\subsubsection*{Definition}
Let \(G\) be a topological group with group law written multiplicatively.
Let \(M\) be a topological \(G\)\=/module, \ie \(M\) is a topological Abelian group, with group law written additively, equipped with a continuous group action of \(G\) compatible with the group structure of \(M\) and denoted by
\[
	\begin{split}
		G \times M & \ra M\\
		(g, m) & \mapsto g \cdot m.
	\end{split}
\]

For each \(n \in \N\), let \(G^n = \prod_{i = 1}^n G\) be the \(n\)\=/fold Cartesian product of \(G\) endowed with the product topology (and \(G^0 = \{\ast\}\) the single element set).
We endow the set of continuous \(n\)\=/cochains of \(G\) with coefficients in \(M\)
\[
	\Crm^n(G, M) = \CC(G^n,M)
\]
with the compact-open topology.

The next proposition follows from Corollary~\ref{coro:compactopengroup}.
\begin{prop} \label{prop:cochaintopgroup}
	If \(G\) is locally compact and separated, then \(\Crm^n(G,M)\) is a topological Abelian group for all \(n \in \N\).
\end{prop}

\begin{rema}
	Corollary~\ref{coro:compactopengroup} always applies to \(G^0=\{\ast\}\), and the map
	\[
		\begin{split}
			\Crm^0(G,M) = \{ f : \{\ast\} \ra M \} & \riso M \\
			f & \mapsto f(\ast)
		\end{split}
	\]
	is an isomorphism of topological groups.
\end{rema}

\begin{exem}
	If \(G\) is compact and \(M\) is discrete, then \(\Crm^n(G,M)\) is discrete for all \(n \in \N\).
\end{exem}

For each \(n \in \N\), the coboundary map
\[
		d_n : \Crm^n(G,M) \ra \Crm^{n+1}(G,M)
\]
is a group homomorphism defined by
\[
	\begin{split}
		d_n (f)(g_1, \ldots , g_{n+1}) & = g_1 \cdot f(g_2, \ldots, g_{n+1})\\
		& \quad + \sum_{i=1}^n (-1)^i f(g_1, \ldots, g_i g_{i+1}, \ldots, g_{n+1})\\
		& \quad + (-1)^{n+1}f(g_1,\ldots,g_n),
	\end{split}
\]
where \(f \in \Crm^n(G,M)\) and \((g_1,\dots,g_{n+1}) \in G^{n+1}\).

\begin{prop} \label{prop:coboundcont}
	If \(G\) is locally compact and separated, then, for each \(n \in \N\), the coboundary map \(d_n\) is a morphism of topological Abelian groups.
\end{prop}
\begin{proof}
	Let \(n \in \N\).
	We prove that the coboundary map \(d_n\) is continuous.
	The coboundary map \(d_n\) decomposes as
	\[
		d_n : \Crm^n(G,M) \xra{\Delta} \prod_{i=0}^{n+1} \Crm^n(G,M) \xra{(d_n^{i})_i} \prod_{i=0}^{n+1} \Crm^{n+1}(G,M) \xra{\tilde{\Sigma}} \Crm^{n+1}(G,M)
	\]
	where
	\begin{itemize}
		\item
	\[
		\begin{split}
			\Delta : \Crm^n(G,M) & \ra \prod_{i=0}^{n+1} \Crm^n(G,M)\\
			f & \mapsto (f, f, \ldots, f),
		\end{split}
	\]
	is the diagonal map which is continuous,
\item
	\[
		\begin{split}
			\tilde{\Sigma} : \prod_{i=0}^{n+1} \Crm^{n+1}(G,M) & \ra \Crm^{n+1}(G,M)\\
			(f_i)_i & \mapsto \sum_{i=0}^{n+1} (-1)^i f_i
		\end{split}
	\]
	is the alternating sum which is continuous by Proposition~\ref{prop:cochaintopgroup},
\item
	\[
		\begin{split}
			d_n^0 : \Crm^n(G,M) & \ra \Crm^{n+1}(G,M)\\
			f & \mapsto [(g_j) \mapsto g_1 \cdot f(g_2, \ldots, g_{n+1})]
		\end{split}
	\]
	which is continuous by Corollary~\ref{coro:coboundpart} (applied to \(X=G\), \(X^\prime=G^n\), \(Y=Z=M\) and \(\psi\) the action of \(G\) on \(M\)),
\item
	\[
		\begin{split}
			d_n^{n+1} : \Crm^n(G,M) & \ra \Crm^{n+1}(G,M)\\
			f & \mapsto [ (g_j) \mapsto f(g_1,\ldots,g_n)]
		\end{split}
	\]
	which is continuous by functoriality (Lemma~\ref{lemm:functorialityCc} applied to \(Y=Y^\prime = M\), \(\psi\) the identity, and \(\phi: X=G^{n+1} \ra X^\prime = G^n, (g_1, \ldots, g_{n+1}) \mapsto (g_1,\ldots , g_n)\))
\item and, for \(1\leq i\leq n\),
	\[
		\begin{split}
			d_n^i : \Crm^n(G,M) & \ra \Crm^{n+1}(G,M)\\
			f & \mapsto [ (g_j) \mapsto f(g_1, \ldots, g_i g_{i+1}, \ldots, g_{n+1})]
		\end{split}
	\]
	which is continuous by functoriality (Lemma~\ref{lemm:functorialityCc} applied to \(Y=Y^\prime = M\), \(\psi\) the identity, and \(\phi: X=G^{n+1} \ra X^\prime = G^n, (g_1, \ldots, g_{n+1}) \mapsto (g_1,\ldots, g_i g_{i+1},\ldots , g_{n+1})\)).
	\end{itemize}
	Hence, the map \(d_n\) is continuous.
	Therefore, by Proposition~\ref{prop:cochaintopgroup}, the coboundary map \(d_n\) is a morphism of topological Abelian groups.
\end{proof}

Let \(\Zrm^n(G,M) = \Ker(d_n)\) be the subgroup of \(\Crm^n(G,M)\) of continuous \(n\)\=/cocycles of \(G\) with coefficients in \(M\), and let
\[
	\Brm^n(G,M) = \left\{ \begin{array}{ll}
		0 & \text{if } n=0,\\
		\Img(d_{n-1}) & \text{if } n>0,
	\end{array}\right.
\]
be the subgroup of continuous \(n\)\=/coboundaries of \(G\) with coefficients in \(M\).
The coboundary maps satisfy \(d_{n+1} \circ d_n = 0\) for each \(n \in \N\), hence, we have
\[
	\Brm^n(G,M) \subset \Zrm^n(G,M) \subset \Crm^n(G,M).
\]
The \(n\)\=/th continuous group cohomology group of \(G\) with coefficients in \(M\) is
\[
	\H^n(G,M) = \Zrm^n(G,M)/\Brm^n(G,M).
\]
We endow \(\Brm^n(G,M)\) and \(\Zrm^n(G,M)\) with the subspace topology from
\(\Crm^n(G,M)\), and \(\H^n(G,M)\) with the quotient topology.

As corollary of Proposition~\ref{prop:cochaintopgroup}, we have:
\begin{coro} \label{coro:cohotopgroup}
	If \(G\) is locally compact and separated, then \(\Brm^n(G,M)\), \(\Zrm^n(G,M)\) and \(\H^n(G,M)\) are topological Abelian groups for all \(n \in \N\).
\end{coro}

\begin{rema}
	The identification of \(\H^0(G,M)\) with the subgroup of \(G\)\=/invariant
	elements \(M^G\) of \(M\), via the map \(f \mapsto f(\ast)\), is an
	isomorphism of topological groups (where \(M^G \subset M\) is equipped with
	the subspace topology from \(M\)).
\end{rema}

\subsubsection*{Functoriality}
If \(\phi: H \ra G\) is a morphism of topological groups, if \(M\) is a topological \(G\)-module and \(N\) is a topological \(H\)-module, and if \(\psi: M \ra N\) is a morphism of topological groups compatible with \(\phi\) (\ie \(h\cdot\psi(m)=\psi(\phi(h)\cdot m)\) for all \(m \in M\) and \(h \in H\)), then \(\psi\) induces a group homomorphism
\begin{equation} \label{eq:funct}
	\H^n(G,M) \ra \H^n(H,N)
\end{equation}
for each \(n \in \N\).
It follows from the functoriality of \(\CC_c(\cdot,\cdot)\) and the definition of the continuous group cohomology that the maps~\eqref{eq:funct} is continuous for each \(n \in \N\).
Therefore, if \(G\) and \(H\) are locally compact and separated, then the maps~\eqref{eq:funct} is a morphism of topological Abelian groups for each \(n \in \N\).

We shall now check that for locally compact and separated topological groups, the topology on the cohomology groups is well-behaved with respect to long exact sequences.
To do so, we shall need the following topological version of the snake lemma.

\begin{lemm}[topological snake lemma] \label{lemm:topsnake}
	Let
	\[
		\begin{tikzcd}
					 & A \ar{r} \ar{d}{\alpha}       & B \ar{r}  \ar{d}{\beta}     & C \ar{r} \ar{d}{\gamma} & 0 \\
			0 \ar{r} & A^\prime \ar{r} & B^\prime \ar{r} & C^\prime & 
		\end{tikzcd}
	\]
	be a commutative diagram of topological Abelian groups whose rows, considered as sequences of Abelian groups, are exact.
	Assume that the topology of \(C\) is the quotient topology from \(B\) and that the topology of \(A^\prime\) is the subspace topology from \(B^\prime\).
	If the kernels of \(\alpha\), \(\beta\) and \(\gamma\) are endowed with their respective subspace topology and the cokernels with their respective quotient topology, then there is a sequence of topological Abelian groups
	\[
		\Ker(\alpha) \ra \Ker(\beta) \ra \Ker(\gamma) \ra \Coker(\alpha) \ra \Coker(\beta) \ra \Coker(\gamma)
	\]
	whose underlying sequence of Abelian groups is exact.
\end{lemm}
\begin{proof}
	A direct proof via ``diagram chasing'' can be found in~\cite[Proposition~4]{Schochet99}.
	We could also argue that the category of topological Abelian groups is \emph{homological}, therefore the ``snake lemma'' holds in it~\cite{BorceuxBourn04,BorceuxClementino05}.
\end{proof}

Let \(G\) be a topological group.
Let
\[
	0 \ra M^\prime \ra M \ra M^\dprime \ra 0
\]
be a short exact sequence of topological \(G\)\=/modules such that the topology of \(M^\prime\) is the subspace topology from \(M\) and the topology of \(M^\dprime\) is the quotient topology from \(M\).
Then, there is a six-term exact sequence of Abelian groups
\begin{equation} \label{eq:h0h1}
	\begin{tikzcd}
		0 \ar{r} & \H^0(G, M^\prime) \ar{r} & \H^0(G, M) \ar{r} \ar[phantom, ""{coordinate, name=Z}]{d} & \H^0(G, M^\dprime) \ar[rounded corners, to path={ -- ([xshift=2ex]\tikztostart.east) |- (Z) [near end]\tikztonodes -| ([xshift=-2ex]\tikztotarget.west) -- (\tikztotarget)}]{dll} \\
				& \H^1(G, M^\prime) \ar{r} & \H^1(G, M) \ar{r} & \H^1(G, M^\dprime).
	\end{tikzcd}
\end{equation}
Moreover, if there exists a continuous section of the projection of \(M\) on \(M^\dprime\) as topological spaces, then there is a long exact sequence of Abelian groups
\begin{equation} \label{eq:hlong}
	\begin{tikzcd}
		\cdots \ar{r} & \H^n(G, M^\prime) \ar{r} & \H^n(G, M) \ar{r} \ar[phantom, ""{coordinate, name=Z}]{d} & \H^n(G, M^\dprime) \ar[rounded corners, to path={ -- ([xshift=2ex]\tikztostart.east) |- (Z) [near end]\tikztonodes -| ([xshift=-2ex]\tikztotarget.west) -- (\tikztotarget)}]{dll} \\
				& \H^{n+1}(G, M^\prime) \ar{r} & \H^{n+1}(G, M) \ar{r} & \H^{n+1}(G, M^\dprime) \ra \cdots.
	\end{tikzcd}
\end{equation}

\begin{rema}
	Such a continuous section of the projection of \(M\) on \(M^\dprime\) exists in particular if the topology of \(M^\dprime\) is the discrete topology.
\end{rema}

\begin{prop} \label{prop:longcont}
	If \(G\) is locally compact and separated, then the sequences~\eqref{eq:h0h1} and~\eqref{eq:hlong} are sequences of topological Abelian groups.
\end{prop}
\begin{proof}
	We recall how the exact sequences of groups~\eqref{eq:h0h1} and~\eqref{eq:hlong} are obtained.
	By functoriality (and under the additional assumption that there exists a continuous section of the projection of \(M\) on \(M^\dprime\) if \(n \geq 1\)), there is a commutative diagram of Abelian groups
	\begin{equation} \label{eq:snakecohom}
		\begin{tikzcd}
			& \frac{\Crm^n(G,M^\prime)}{\Brm^n(G,M^\prime)} \ar{r} \ar{d} & \frac{\Crm^n(G,M)}{\Brm^n(G,M)} \ar{r} \ar{d} & \frac{\Crm^n(G,M^\dprime)}{\Brm^n(G,M^\dprime)} \ar{d} \ar{r} & 0 \\
			0 \ar{r} & \Zrm^{n+1}(G,M^\prime) \ar{r} & \Zrm^{n+1}(G,M) \ar{r} & \Zrm^{n+1}(G,M^\dprime) &
		\end{tikzcd}
	\end{equation}
	where the vertical arrows are the factorisations of the coboundary maps (and where the surjectivity of \(\Crm^n(G,M) \ra \Crm^n(G,M^\dprime)\) follows from Lemma~\ref{lemm:relativetop} when \(n \geq 1\)).
	The exact sequences of groups~\eqref{eq:h0h1} and~\eqref{eq:hlong} are obtained by applying the classical snake lemma to the diagram~\eqref{eq:snakecohom}.

	By Proposition~\ref{prop:cochaintopgroup} and Corollary~\ref{coro:cohotopgroup}, each object in the diagram~\eqref{eq:snakecohom} is a topological group, and we already noted that the horizontal arrows in the diagram~\eqref{eq:snakecohom} are all morphisms of topological groups by functoriality of \(\CC_c(\cdot,\cdot)\).
	By Proposition~\ref{prop:coboundcont}, the vertical arrows in the diagram~\eqref{eq:snakecohom} which are the factorisations of the coboundary maps are also morphisms of topological groups.
	Moreover, by Lemma~\ref{lemm:relativetop}, the topology of \(\Crm^n(G,M^\dprime)/\Brm^n(G,M^\dprime)\) is the quotient topology from \(\Crm^n(G,M)/\Brm^n(G,M)\), and the topology of \(\Zrm^{n+1}(G,M^\prime)\) is the subspace topology from \(\Zrm^{n+1}(G,M)\).
	Therefore, the topological snake lemma~\ref{lemm:topsnake} applies to the diagram~\eqref{eq:snakecohom} and implies that the sequences~\eqref{eq:h0h1} and~\eqref{eq:hlong} are sequences of topological Abelian groups.
\end{proof}

\bibliographystyle{plain}
\bibliography{references}
\contact
\end{document}